\documentclass[a4paper]{amsart}

\usepackage[english]{babel}
\usepackage[utf8x]{inputenc}
\usepackage[T1]{fontenc}
\usepackage{amsthm}
\usepackage{amsmath}
\theoremstyle{plain}
\usepackage{chngcntr}
\usepackage{relsize}

\newtheorem{Lemma}{Lemma}
\newtheorem{Theorem}[Lemma]{Theorem}

\newtheorem{Corollary}[Lemma]{Corollary}

\newtheorem{Definition}[Lemma]{Definition}

\usepackage[a4paper,top=3cm,bottom=2cm,left=3cm,right=3cm,marginparwidth=1.75cm]{geometry}

\usepackage{float}
\usepackage{nccmath}
\usepackage{mathtools}
\usepackage{amsfonts}
\usepackage{amsmath}
\usepackage{graphicx}
\usepackage[colorinlistoftodos]{todonotes}
\usepackage[colorlinks=true, allcolors=blue]{hyperref}

\title{\textbf{Diophantine Approximation with Prime Restriction in Real Quadratic Number Fields}}
\date{\today}

\author{Stephan~Baier}
\address{Stephan~Baier\\%
	Ramakrishna Mission Vivekananda Educational Research Institute\\%
	Department of Mathematics\\%
	G.\ T.\ Road, PO~Belur Math, Howrah, West Bengal~711202\\%
	India}
\email{stephanbaier2017@gmail.com}

\author{Dwaipayan~Mazumder}
\address{Dwaipayan~Mazumder\\%
	Ramakrishna Mission Vivekananda Educational Research Institute\\%
	Department of Mathematics\\%
	G.\ T.\ Road, PO~Belur Math, Howrah, West Bengal~711202\\%
	India}
\email{dwmaz.1993@gmail.com}

\begin{document}
\begin{abstract}
	The distribution of $\alpha p$ modulo one, where $p$ runs over the rational primes and $\alpha$ is a fixed irrational real, has received a lot of attention. It is natural to ask for which exponents $\nu>0$ one can establish the infinitude of primes $p$ satisfying $||\alpha p||\le p^{-\nu}$. The latest record in this regard is Kaisa Matom\"aki's landmark result $\nu=1/3-\varepsilon$ which presents the limit of currently known technology. Recently, Glyn Harman, and, jointly, Marc Technau and the first-named author, investigated the same problem in the context of imaginary quadratic fields. Glyn Harman obtained an analog for $\mathbb{Q}(i)$ of his result in the context of $\mathbb{Q}$, which yields an exponent of $\nu=7/22$. Marc Technau and the first-named author produced an analogue of Bob Vaughan's result $\nu=1/4-\varepsilon$ for all imaginary quadratic number fields of class number 1. In the present article, we establish an analog of the last-mentioned result for real quadratic fields of class number 1 under a certain Diophantine restriction. This setting involves the additional complication of an infinite group of units in the ring of integers. Moreover, although the basic sieve approach remains the same (we use an ideal version of Harman's sieve), the problem takes a different flavor since it becomes truly 2-dimensional. We reduce it eventually to a counting problem which is, interestingly, related to roots of quadratic congruences. To approximate them, we use an approach by Christopher Hooley based on the theory of binary quadratic forms.  
\end{abstract}
\maketitle
\tableofcontents

\section{Introduction and main results}
\subsection{History}
Dirichlet's approximation theorem for the rationals implies that for every $x_0\in \mathbb{R}\setminus \mathbb{Q}$, there are infinitely many pairs $(p,q)$ of coprime integers $p$ and $q$ such that 
$$
\left|x_0-\frac{p}{q}\right|\le q^{-2}.
$$ 
Whereas the above statement is easy to prove just using pigeonhole principle or the continued fraction expansion of $x_0$, the problem of Diophantine approximation by fractions with denominator restricted to primes is hard. This problem has a long history and triggered the development of important tools in analytic number theory. It is not difficult to prove using the generalized Riemann Hypothesis for Dirichlet $L$-functions that for every $\varepsilon>0$ and $\nu=1/3$, there are infinitely many pairs $(p,\pi)$ with $p$ an integer and $\pi$ a prime such that
\begin{equation} \label{primeapprox}
\left|x_0-\frac{p}{\pi}\right|\le \pi^{-1-\nu+\varepsilon},
\end{equation}
but an improvement beyond $\nu=1/3$ depends on very strong assumptions on primes in arithmetic progressions. One might expect that \eqref{primeapprox} holds for infinitely many pairs $(p,\pi)$ if $\nu=1$. Important unconditional results started with Vinogradov \cite{Vino} who showed that $\nu=1/5$ is admissible using his intricate non-trivial treatment of trigonometrical sums over primes. Vaughan \cite{Vau} simplified this treatment introducing his famous identity for sums over primes and improved the exponent to $\nu=1/4$ by refining Fourier-analytic arguments in Vinogradov's method. Harman \cite{Har1} developed a new sieve method which enabled him to establish the exponent $\nu=3/10$. These results were subsequently improved by Jia and Harman in several papers. In particular, Harman got the exponent $\nu=7/22$ in \cite{Har2}. Heath-Brown and Jia \cite{HeJia} brought Kloosterman sums into the picture and made a number of further innovations to reach $\nu=16/49$ which falls short of the exponent $1/3$. Finally, in a landmark paper, Matom\"aki \cite{Ma} managed to reach $\nu=1/3$ by using bounds for averages of Kloosterman sums. This is considered the limit of the current technology.

It is interesting to put the said problem on restricted Diophantine approximation in the framework of number fields. Novel ideas are required to make the classical methods work in this context, and so far there are only a few recent results in this regard. The first-named author \cite{Bai} extended the classical problem, in slightly generalized form, to $\mathbb{Q}(i)$ and his method led to an exponent of $\nu=1/12$.  Harman \cite{Har2} established the full analog of his above-mentioned result with $\nu=7/22$ for $\mathbb{Q}(i)$ by introducing a number of novelties into his method, in particular, a clever estimation of trigonometrical sums over 
regions of $\mathbb{C}$ with sharp cutoff, which are more difficult to handle than trigonometrical sums over intervals in $\mathbb{R}$ since the geometry of the 
regions comes into play. The first-named author and Technau \cite{BaiTech} considered the problem for all imaginary quadratic number fields of class number 1 and 
obtained an exponent corresponding to $\nu=1/4$. They avoided the said trigonometrical sums with sharp cutoff by using a smoothed version of Harman's sieve for 
imaginary quadratic fields and Poisson summation to transform smooth trigonometrical sums over the entire complex plane. Their result can be put in the following form. 

\begin{Theorem} \label{Bai-Tech}
Assume that $\mathbb{Q}(\sqrt{-d})$ with $d>1$ a square-free positive integer has class number 1. Let $x_0 \in \mathbb{C}\setminus \mathbb{Q}(\sqrt{-d})$.
Then there exist infinitely many non-zero prime ideals $\mathfrak{p}$ in $\mathcal{O}$, the ring of integers of $\mathbb{Q}(\sqrt{-d})$,  
such that 
\begin{equation*}
\left|x_0-\frac{p}{q}\right|\le \mathcal{N}(\mathfrak{p})^{-1/2-1/8+\varepsilon} 
\end{equation*}
for some generator $q$ of $\mathfrak{p}$ and $p\in \mathcal{O}$, where 
$\mathcal{N}(\mathfrak{p})$ denotes the norm of $\mathfrak{p}$. 
\end{Theorem}

In simpler words, under the assumptions of the above theorem, there exist
infinitely many prime elements $\pi\in \mathcal{O}$ such that
\begin{equation*}
\left|x_0-\frac{p}{\pi}\right|\le |\pi|^{-1-1/4+\varepsilon}
\end{equation*}
for a suitable $p\in \mathcal{O}$. This corresponds to Vaughan's exponent $\nu=1/4$ in the classical setting. In this article, we prove an analog of Theorem \ref{Bai-Tech} for 
real quadratic number fields of class number 1 under a certain Diophantine restriction to the pairs $(x_1,x_2)\in \mathbb{R}^2$ replacing $x_0$ in this context. We will see that, in the sense of the Lebesgue measure, almost all of these $(x_1,x_2)$ will satisfy this restriction. In subsection \ref{GB}, we shall elaborate more about it. 

The real quadratic setting is more difficult and of different flavor because the problem becomes truly two-dimensional and we have to handle the infinite unit group. Therefore, we set up the problem in the context of ideals from the very beginning using weight functions on the ideals which themselves are smoothed sums over the set of their generators. Later, these sums are unfolded in order to make convenient use of Poisson summation. The ideal setup allows us to handle the infinite unit group. By the said smoothing, we avoid, similarly as in \cite{BaiTech}, unpleasant two-dimensional trigonometrical sums with sharp cutoff.

As in the setting of $\mathbb{Q}$ or $\mathbb{Q}(i)$ (or more generally, imaginary quadratic fields of class number 1), we need to introduce a Diophantine approximation to bound certain averages of trigonometrical sums at some stage. In the context of real quadratic fields, this is a simultaneous Diophantine approximation of our pair $(x_1,x_2)$ by a pair of conjugates in $\mathbb{Q}(\sqrt{d})$. We note that a simple approximation by a pair of rationals with the same denominator using a two-dimensional version of the Dirichlet approximation theorem turns out to be insufficient for our purposes (see subsection \ref{DA}).  We will be led to a counting problem for solutions of two-dimensional systems of linear congruences whose resolution depends, interestingly, on information about the Diophantine properties of roots of quadratic congruences, 
as established by Hooley \cite{Hoo} making use of the theory of quadratic forms. This is an interesting new feature which is not present in the imaginary-quadratic case. 

Our starting point will again be a smoothed version of Harman's sieve for quadratic number fields, where we here use a formulation with ideals 
instead of algebraic integers. Before we state our main theorem, we introduce some notations which will be used throughout this article and review Dirichlet approximation in real quadratic number fields.

\subsection{Notations}
\begin{itemize}
\item We assume that $d$ is a positive square-free integer satisfying $d\equiv 3
\bmod{4}$ in which case the ring of integers of $\mathbb{Q}(\sqrt{d})$ equals 
$\mathbb{Z}[\sqrt{d}]$.
\item We denote the ring of integers of $\mathbb{Q}(\sqrt{d})$ by $\mathcal{O}$ and write $\mathbb{K}:=\mathbb{Q}(\sqrt{d})$.
\item We denote the set of ideals in $\mathcal{O}$ by $\mathcal{I}$.
\item We denote the two embeddings of $\mathbb{Q}(\sqrt{d})$, given by the identity and conjugation, by 
$$
\sigma_1(\alpha+\beta\sqrt{d}):=\alpha+\beta\sqrt{d}
$$ 
and 
$$
\sigma_2(\alpha+\beta\sqrt{d}):=\alpha-\beta\sqrt{d}.
$$
\item We write $\sigma(\mathbb{K}):=\{(\sigma_1(\gamma),\sigma_2(\gamma)) : \gamma\in K\}$.
\item We assume that $\mathbb{Q}(\sqrt{d})$ has class number 1 so that all ideals in $\mathcal{O}$ are principal.
\item The norm of an ideal $\mathfrak{q}\in \mathcal{O}$ will be denoted by 
$\mathcal{N}(\mathfrak{q})$.
\item We write $\mathcal{N}(q)$ for the {\it modulus} of the norm over $\mathbb{Q}$ of an algebraic integer $q\in \mathcal{O}$, i.e. $\mathcal{N}(q)=\mathcal{N}((q))$.  
\item If $p,q\in \mathcal{O}$, we write gcd$(p,q)\approx t$ to mean that $t$ is a greatest common divisor of $p$ and $q$ in $\mathcal{O}$. We note that $t$ is unique up to units in $\mathcal{O}$. 
\item $(x_1,x_2)$ is a pair of real numbers which does not belong to 
$\sigma(\mathbb{K})$. 
\item $N$, $x$ and $\delta$ will be variables, where $N>1$ is a natural number, $x>1$
is a real number and $N^{-1}\le \delta\le 1$.
\item $\varepsilon$ is an arbitrary but fixed positive number. 
\item As usual, we write 
$f=O(g)$ or $f\ll g$ if the functions $f$ and $g$ satisfy $|f|\le c|g|$ for some positive constant $c$, and we write $f\asymp g$ if $f\ll g$ and $g\ll f$.
\item We allow all $O$-constants to depend on $\varepsilon$, $d$ and $(x_1,x_2)$.
\end{itemize}

\subsection{Dirichlet approximation in real quadratic number fields}
A general version of Dirichlet's approximation theorem for number fields was given in \cite[Theorem 1]{Que}. This implies the following version for the case of real quadratic number 
fields $\mathbb{Q}(\sqrt{d})$, where no congruence conditions on $d$ and no conditions on the class number are required here.

\begin{Theorem} \label{diriapptheoremstrong} There exists a constant $C>0$ with the following property. If $(x_1,x_2)\in \mathbb{R}^2\setminus \sigma(\mathbb{K})$, then there are infinitely many pairs 
$(p,q)\in \mathcal{O}\times (\mathcal{O}\setminus \{0\})$ such that 
$$
\left|x_i-\frac{\sigma_i(p)}{\sigma_i(q)}\right|\le \frac{C}{|\sigma_i(q)|(|\sigma_1(q)|+|\sigma_2(q)|)} \quad \mbox{ for } i=1,2.
$$
\end{Theorem}  

Since $|\sigma_i(q)| (|\sigma_1(q)|+|\sigma_2(q)|)\ge |\sigma_1(q)\sigma_2(q)|=\mathcal{N}(q)$ for $i=1,2$ and $\mathcal{N}(\tilde{q})\ge \mathcal{N}(q)$ whenever $\tilde{p}/\tilde{q}=p/q$ and
$p$ and $q$ are coprime in $\mathcal{O}$, the following is an immediate Corollary. 

\begin{Corollary} \label{diriapptheoremweak} There exists a constant $C>0$ with the following property. If $(x_1,x_2)\in \mathbb{R}^2\setminus \sigma(\mathbb{K})$, then there are infinitely many principal ideals 
$\mathfrak{q}\in I\setminus 0$ such that 
\begin{equation*}
\left|x_i-\frac{\sigma_i(p)}{\sigma_i(q)}\right|\le \frac{C}{\mathcal{N}(\mathfrak{q})} \quad \mbox{ for } i=1,2
\end{equation*}
for some generator $q$ of $\mathfrak{q}$ and $p\in \mathcal{O}$ coprime to $q$. 
\end{Corollary}

Our main result on Diophantine approximation with prime restriction in $\mathbb{Q}(\sqrt{d})$ will depend on a certain Diophantine property of $(x_1,x_2)$ which we introduce 
next.

\subsection{Good and bad $(x_1,x_2)$} \label{GB}
By Corollary \ref{diriapptheoremweak}, there are infinitely many natural numbers $W$ such that 

\begin{equation} \label{theabove}
\left|x_i- \frac{\sigma_i(u+v\sqrt{d})}{\sigma_i(f+g\sqrt{d})}\right|\le \frac{1}{\mathcal{N}(f+g\sqrt{d})} \quad \mbox{ for }
i=1,2
\end{equation}
for suitable $u,v,f,g\in \mathbb{Z}$ with $u+v\sqrt{d}$ and $f+g\sqrt{d}$ coprime in $\mathcal{O}$ and $\mathcal{N}(f+g\sqrt{d})=W$. 
We shall require the following notion.

\begin{Definition} \label{goodbad} For $\eta>0$, we say that  $(x_1,x_2)\in \mathbb{R}^2\setminus \sigma(\mathbb{K})$ is $\eta$-good if there is an infinite sequence of natural numbers $W$ such that \eqref{theabove} holds 
with  
$$
{\mathcal{N}(f+g\sqrt{d})}=W, \quad \mbox{\rm gcd}(u+v\sqrt{d},f+g\sqrt{d})\approx 1 \quad \mbox{and} \quad 
{\rm gcd}(f,g)=O\left(W^{\eta}\right),
$$
where {\rm gcd}$(f,g)$ is meant to be the largest natural number dividing both $f$ and $g$. 
We call $(x_1,x_2)$ good if it is $\eta$-good for all $\eta>0$. We call $(x_1,x_2)$ bad if it is not good.
\end{Definition}

We shall obtain the full analog of Vaughan's classical 1/4-result for {\it good} pairs. This constraint is somewhat unsatisfactory, but we shall show in subsection \ref{almostall} that, in the sense of the Lebesgue measure, almost all $(x_1,x_2)\in \mathbb{R}^2\setminus \sigma(\mathbb{K})$ are good. In subsection \ref{consgood}, we shall supply an explicit construction of particular good pairs $(x_1,x_2)$. At this stage we are, however, not able to decide whether there are bad pairs $(x_1,x_2)$ {\it at all} or if they are just an artifact. It is well possible that the set of bad pairs is actually empty. We pose this as an open problem, left to future research. \\ \\
{\bf Problem:} Decide whether there are bad pairs $(x_1,x_2)$ or not. If there are, describe their properties and find a different method which allows to treat them effectively, with the goal of obtaining the full analog of Vaughan's classical result for {\it all} $(x_1,x_2) \in \mathbb{R}^2\setminus \sigma(\mathbb{K})$. \\ \\
The point where the relevant parameter $Z={\rm gcd}(f,g)$ comes into play is in subsection \ref{additive}. Here a linear congruence \eqref{singlecon'} to a modulus of the form $W'Z$ appears, whereas another important relation, the quadratic congruence \eqref{congrurel'}, does not have an extra factor of $Z$ in the modulus. We are not able to handle the factor $Z$ in the modulus $W'Z$ of the congruence \eqref{singlecon'} and have no choice but to throw it away, which causes a loss. The resulting congruence to modulus $W'$, however, can be successfully treated in combination with the said quadratic congruence \eqref{congrurel'} to the same modulus $W'$.    

\subsection{Main result}
The following is our main result. 

\begin{Theorem} \label{mainresult}
Assume that $\mathbb{Q}(\sqrt{d})$ has class number 1, where $d$ is a positive square-free integer satisfying $d\equiv 3 \bmod{4}$. Let $\varepsilon$ be any positive real number. Suppose further that $(x_1,x_2)\in \mathbb{R}^2\setminus \sigma(\mathbb{K})$ is $\eta$-good
in the above sense. Set
$$
\nu:=\frac{1/8-\eta}{1+2\eta}.
$$
Then there exist infinitely many non-zero prime ideals $\mathfrak{p}$ in the ring $\mathcal{O}$ of integers of $\mathbb{Q}(\sqrt{d})$
such that 
\begin{equation*}
\left|x_i-\frac{\sigma_i(p)}{\sigma_i(q)}\right|\le \mathcal{N}(\mathfrak{p})^{-1/2-\nu+\varepsilon} \quad \mbox{ for } i=1,2
\end{equation*}
for some generator $q$ of $\mathfrak{p}$ and $p\in \mathcal{O}$. 
If $(x_1,x_2)$ is good, then the above holds with $\nu=1/8$. 
\end{Theorem}

In particular, for a concrete subset of $\mathbb{R}^2$ of full Lebesgue measure, the good $(x_1,x_2)$, we have the real quadratic analog to Theorem \ref{Bai-Tech}.
We note that Theorem \ref{mainresult} gives a nontrivial estimate except when $(x_1,x_2)$ is not $\eta$-good for any $\eta<1/8$. The set of these $(x_1,x_2)$ has Lebesgue measure 0 since it is a subset of the set of 
bad $(x_1,x_2)$. 

Below are some comments on our conditions ``$d\equiv 3 \bmod{4}$'' and ``class number 1'' in Theorem \ref{mainresult}. First, we point out that in contrast to the imaginary-quadratic case where one has only finitely many fields of class number 1 by the celebrated Baker-Heegner-Stark Theorem, it is conjectured that there exist infinitely many real quadratic number fields $\mathbb{Q}(\sqrt{d})$ of class number 1, and this should remain true when $d$ is restricted to integers congruent to 3 modulo 4. 

In fact, the restriction to $d\equiv 3 \bmod{4}$ is non-essential and just made for convenience because under this condition we can write all elements of $\mathcal{O}$ in the form $a+b\sqrt{d}$ with $a,b$ integers, and we have that $d$ is odd, which will turn out convenient in subsection \ref{redu} but is not  needed anywhere else. It should not cause 
much trouble to establish Theorem \ref{mainresult} also for the cases when $d\equiv 1,2\bmod{4}$ along the same lines. 

The condition of $\mathbb{Q}(\sqrt{d})$ having class number 1 makes the proof convenient but it should not be hard to remove this assumption as well. Essentially, what needs to be changed to settle the case when the class number is greater than 1 is to restrict the prime ideals in the main results to principal prime ideals and to assume that the sieve weights $\omega(\mathfrak{q})$ and $\tilde{\omega}(\mathfrak{q})$, defined in section 2, are supported on principal ideals. The sieve itself remains the same (formulated for general weight functions on {\it all} integral ideals). The important quantity $\mathcal{T}(N)$, defined in equation \eqref{2.5}, now measures the number of principal prime ideals with norm of size about $N$, scaled by a factor of $\delta^2/(2\sqrt{d})$. This quantity has to satisfy the same lower bound $\mathcal{T}(N)\gg \delta^2 N/\log N$ as in the class number 1 setting, which is indeed the case since a positive proportion of prime ideals is principal. (To see this, single out the principal prime ideals from the prime ideals using class group characters and use the Hecke $L$-functions associated to them.)

Moreover, it would also be desirable to improve the exponent $\nu$ in our main result. To this end, one would need to replace the asymptotic Harman sieve for quadratic number fields by a lower bound sieve similar to that used by Harman in the cases of $\mathbb{Q}$ and $\mathbb{Q}(i)$. To work out such a lower bound sieve for number fields seems feasible. However, looking at its proof in the classical setting of $\mathbb{Q}$, it depends on asymptotic estimates for averages of the weight function over primes (or more generally, multiples of primes). Considering the treatment in \cite{BaiTech}, asymptotics of this kind seem to be available in the imaginary-quadratic case. The weight functions used in this article on the real quadratic case are more complicated, though, and it is not immediately clear if they are suitable to obtain the required asymptotics. To settle these technical issues presents another goal of future research. \\ \\ 
{\bf Acknowledgements.} The authors would like to thank the referee for his detailed comments and detecting some inaccuracies in the original version of this article. Further, we would like to thank the Ramakrishna Mission Vivekananda Educational and Research Institute for providing an excellent work environment. The second-named author's research has been supported by a UGC Net fellowship.

\section{Smoothed setup}
We begin with smoothing our Diophantine problem. Throughout the sequel, we let $\mathcal{C}$ be a natural number which will be fixed in the course of this article 
depending on $\varepsilon$ and no other parameter. We shall make
use of the non-negative function
\begin{equation} \label{fdef}
f(x)=\left(\exp(-\pi x^2)-\exp(-2 \pi x^2)\right)^{\mathcal{C}},
\end{equation}
which has, on the one hand, exponential decay as $|x|\rightarrow \infty$, and satisfies, on the other hand, the bound
\begin{equation} \label{useful}
f(x) \ll_{\mathcal{C}} \min\left\{1,|x|^{2\mathcal{C}}\right\}
\end{equation}
for all real $x$, which is strong if $x$ is small. We also write
\begin{align} \label{Omegadef}
    \Omega_{\Delta}(x):=\exp\left(-\pi \cdot \frac{x^2}{\Delta^2}\right).
\end{align}

We define two functions $\omega,\tilde{\omega}:\mathcal{I}\longrightarrow \mathbb{R}_{\ge 0}$ as 
\begin{align} \label{2.1}
    \omega(\mathfrak{q}):=\frac{\delta^2}{2\sqrt{d}}\cdot \Psi(\mathfrak{q}) 
\end{align}
and 
\begin{align} \label{2.2}
    \tilde{\omega}(\mathfrak{q}):=\tilde{\Psi}(\mathfrak{q})\cdot  F(\mathfrak{q}) 
\end{align}
with
\begin{equation} \label{Psidef}
\begin{split}
\Psi(\mathfrak{q}):=& \sum\limits_{\substack{k\in \mathcal{O}\\ (k)=\mathfrak{q}}} f\left(\frac{\sigma_1(k)}{\sqrt{N}}\right)f\left(\frac{\sigma_2(k)}{\sqrt{N}}\right), \quad 
\tilde{\Psi}(\mathfrak{q}):=\frac{N}{\mathcal{N}(\mathfrak{q})}\cdot \Psi(\mathfrak{q})
\end{split}
\end{equation}
and
\begin{equation} \label{Fdef}
F(\mathfrak{q}):=\sum\limits_{p\in\mathcal{O}}\Omega_{\delta/\sqrt{N}}\left(x_1-\frac{\sigma_1(p)}{\sigma_1(q)}\right)\Omega_{\delta/\sqrt{N}}\left(x_2-\frac{\sigma_2(p)}{\sigma_2(q)}\right),
\end{equation}
where $q$ in \eqref{Fdef} is any generator of $\mathfrak{q}$, i.e. 
\begin{align*}
  \mathfrak{q}=(q).  
\end{align*}
Below we will see that $\omega(\mathfrak{q})$ and $\tilde{\omega}(\mathfrak{q})$
are well-defined. 

It is easy to see that the sum over $p$ in \eqref{Fdef} converges. This is because  
\begin{equation} \label{lattice}
\Lambda(q;x_1,x_2)=\left\{\left(x_1-\frac{\sigma_1(p)}{\sigma_1(q)}, x_2-\frac{\sigma_2(p)}{\sigma_2(q)}\right) \in \mathbb{R}^2: p\in \mathcal{O}\right\}
\end{equation}
is a shifted lattice in $\mathbb{R}^2$, and $\Omega_{\Delta}(x)$ 
is exponentially decreasing for $|x|\rightarrow \infty$. Moreover, we have the upper bound 
\begin{equation} \label{Fsize}
F(\mathfrak{q})\ll 1+\frac{\delta^2}{N}\cdot \mathcal{N}(\mathfrak{q}),
\end{equation}
which will be provided in section \ref{Psiev}.

The convergence of the sum over $q$ in \eqref{Psidef} is easy to see as well, and moreover we have the upper bounds
\begin{equation} \label{Psisize}
 \Psi(\mathfrak{q}),\tilde{\Psi}(\mathfrak{q})\ll 
\exp\left(-\pi \mathcal{D}\mathcal{C}\cdot \frac{\mathcal{N}(\mathfrak{q})}{N}\right)\cdot \log N
\end{equation}
for every ideal $\mathfrak{q}\in \mathcal{I}\setminus\{0\}$ and a suitable constant $\mathcal{D}>0$, which will be proved in section \ref{Psiev}, where we shall also provide the lower bounds
\begin{equation} \label{Psisizeobs}
\Psi(\mathfrak{q}),\tilde{\Psi}(\mathfrak{q})\gg 1 \quad \mbox{ if } N\le \mathcal{N}(\mathfrak{q})\le 2N. 
\end{equation}

We still need to show that $F(\mathfrak{q})$ is independent of the choice of the generator $q$, i.e., the sum on the right-hand side of \eqref{Fdef} is invariant under a change of variables $q \rightarrow uq$, where $u$ is a unit in $\mathcal{O}$.
Putting $uq$ in place of $q$, we get
\begin{gather*}
    \sum\limits_{p\in\mathcal{O}}\Omega_{\delta/\sqrt{N}}\left(x_1-\frac{\sigma_1(p)}{\sigma_1(qu)}\right)\Omega_{\delta/\sqrt{N}}\left(x_2-\frac{\sigma_2(p)}{\sigma_2(qu)}\right)\\
     =\sum\limits_{p\in\mathcal{O}}\Omega_{\delta/\sqrt{N}}\left(x_1-\frac{\sigma_1(pu^{-1})}{\sigma_1(q)}\right)\Omega_{\delta/\sqrt{N}}\left(x_2-\frac{\sigma_2(pu^{-1})}{\sigma_2(q)}\right),
\end{gather*}
which equals $F(\mathfrak{q})$ upon making the change of variables $pu^{-1} \rightarrow p$.

Thus, we have seen that $\Psi(\mathfrak{q})$, $\tilde{\Psi}(\mathfrak{q})$ and $F(\mathfrak{q})$ are well-defined, and the definitions of $\omega(\mathfrak{q})$ and $\tilde{\omega}(\mathfrak{q})$ in \eqref{2.1} and \eqref{2.2} together with the bounds 
\eqref{Fsize} and \eqref{Psisize} and $\delta\le 1$ give
\begin{equation} \label{omegabounds}
\omega(\mathfrak{q})\ll \delta^2\exp\left(-\pi \mathcal{D}\mathcal{C}\cdot \frac{\mathcal{N}(\mathfrak{q})}{N}\right)\cdot \log N \quad \mbox{ and }\quad \tilde{\omega}(\mathfrak{q})\ll \exp\left(-\frac{\pi}{2}\cdot  \mathcal{D}\mathcal{C}\cdot \frac{\mathcal{N}(\mathfrak{q})}{N}\right)\cdot \log N.
\end{equation}

Our goal is to derive a lower bound for the quantity
\begin{align} \label{2.4}
    \tilde{\mathcal{T}}(N):=\sum\limits_{\mathfrak{p}}\tilde{\omega}(\mathfrak{p}),
\end{align}
where the sum on the right-hand side runs over all prime ideals $\mathfrak{p}\in \mathcal{I}\setminus\{0\}$. The convergence of this sum is again ensured due to the bound for $\tilde{\omega}(\mathfrak{q})$ in \eqref{omegabounds}.
The above quantity $\tilde{\mathcal{T}}(N)$
measures, up to units, the number of prime elements $q \in \mathcal{O}$ with norm of size about $N$ such that  $|x_i-\sigma_i(p)/\sigma_i(q)|$ is not much larger than $\delta/\sqrt{N}$ for $i=1,2$ and a suitable $p\in \mathcal{O}$.
Our approach is, following Harman's philosophy, to compare $\tilde{\mathcal{T}}(N)$ with the quantity
\begin{align} \label{2.5}
    \mathcal{T}(N):=\sum\limits_{\mathfrak{p}}\omega(\mathfrak{p}),
\end{align}
where the sum on the right-hand side again runs over all prime ideals $\mathfrak{p}\in \mathcal{I}\setminus\{0\}$. The quantity $\mathcal{T}(N)$ 
measures, up to units, the number of prime elements $q\in \mathcal{O}$ with norm of size about $N$, scaled by a factor of 
$\delta^2/(2\sqrt{d})$. Using Landau's prime ideal theorem together with  \eqref{2.1} and \eqref{Psisizeobs}, we obtain
\begin{equation*} 
\mathcal{T}(N)\gg \delta^2\cdot \frac{N}{\log N}.
\end{equation*}
The main task of this article is to show that, under suitable conditions on $\delta$, the difference
$$
\tilde{\mathcal{T}}(N)-\mathcal{T}(N)
$$
is small compared to $\delta^2N/\log N$ so that the above lower bound for $\mathcal{T}(N)$ yields one for $\tilde{\mathcal{T}}(N)$. 
We shall establish the following.

\begin{Theorem}\label{difftheorem}
Assume that $(x_1,x_2)\in \mathbb{R}^2\setminus \sigma(\mathbb{K})$ is $\eta$-good, where $\eta>0$. Suppose that $\varepsilon\le 1/14$. Then there exist infinitely many natural numbers $N$ such that
$$
\tilde{\mathcal{T}}(N)-\mathcal{T}(N)\ll \delta^2N^{1-\varepsilon}
$$
and hence
$$
\tilde{\mathcal{T}}(N)\gg \delta^2\cdot \frac{N}{\log N},
$$
provided that 
$$ 
1\ge \delta\ge N^{-\nu+15\varepsilon},
$$
where
$$
\nu:=\frac{1/8-\eta}{1+2\eta}.
$$
\end{Theorem}

If $(x_1,x_2)$ is good, we can choose $\eta$ as small as we wish and therefore
get the following as an immediate Corollary. 

\begin{Corollary}\label{difftheorem2}
Assume that $(x_1,x_2) \in \mathbb{R}^2\setminus \sigma(\mathbb{K})$ is good. Suppose that $\varepsilon\le 1/14$. Then there exist infinitely many natural numbers $N$ such that
$$
\tilde{\mathcal{T}}(N)-\mathcal{T}(N)\ll \delta^2N^{1-\varepsilon}
$$
and hence
$$
\tilde{\mathcal{T}}(N)\gg \delta^2\cdot \frac{N}{\log N},
$$
provided that 
$$ 
1\ge \delta\ge N^{-1/8+16\varepsilon}.
$$
\end{Corollary}

In section \ref{unsmoothing}, we shall derive our main result, Theorem \ref{mainresult} from Theorem \ref{difftheorem} and Corollary \ref{difftheorem2}.

\section{Bounds for $\Psi(\mathfrak{q})$, $\tilde{\Psi}(\mathfrak{q})$ and $F(\mathfrak{q})$} \label{Psiev}
In this section, we shall show that the sum on the right-hand side of \eqref{Psidef} converges and that $\Psi(\mathfrak{q})$ and $\tilde{\Psi}(\mathfrak{q})$ satisfy \eqref{Psisize}
and \eqref{Psisizeobs}. We shall also derive the bound \eqref{Fsize} for 
$F(\mathfrak{q})$. To this end, we shall use the following lemma which will also be needed in later parts of this paper. 

\begin{Lemma} \label{sigmasize}
There exist constants $c_1,c_2>0$ depending only on $\mathbb{K}$ with the following property. For every $m_0\in\mathcal{O} \setminus\{0\}$, there exists a unit $u$ in $\mathcal{O}$ such that 
$$
c_1\sqrt{\mathcal{N}(m_0)}\le |\sigma_i(um_0)| \le c_2\sqrt{\mathcal{N}(m_0)} \mbox{ for } i=1,2.
$$
\end{Lemma}

\begin{proof} This follows from a more general result in Minkowski theory (see, for example, \cite[Lemma (6.2), section I.6, page 38]{Neu}). We give a short direct proof below.
 
Let $\epsilon$ be the fundamental unit, i.e. the smallest unit in $\mathcal{O}$ exceeding 1. Then for any $k\in \mathbb{Z}$, we have
$$
\sigma_1\left(\epsilon^km_0\right)=\epsilon^k\sigma_1(m_0)
$$
and 
$$
\sigma_2\left(\epsilon^km_0\right)=\epsilon^{-k}\sigma_2(m_0).
$$
If $\rho$ is a real number satisfying
$$
|\epsilon^{\rho}\sigma_1(m_0)|=|\epsilon^{-\rho}\sigma_2(m_0)|, 
$$
then since
$$
|\epsilon^{\rho}\sigma_1(m_0)|\cdot |\epsilon^{-\rho}\sigma_2(m_0)|=\mathcal{N}(m_0),
$$
it follows that 
$$
|\epsilon^{\rho}\sigma_1(m_0)|=\sqrt{\mathcal{N}(m_0)}=|\epsilon^{-\rho}\sigma_2(m_0)|.
$$
This real number $\rho$ is given by
$$
\rho=\frac{\log|\sigma_2(m_0)|-\log|\sigma_1(m_0)|}{2\log\epsilon}. 
$$
Take $k:=\lfloor \rho \rfloor$. Then it follows that
$$
\epsilon^{-1} \sqrt{\mathcal{N}(m_0)}\le |\sigma_i(\epsilon^km_0)|\le \epsilon \sqrt{\mathcal{N}(m_0)} \mbox{ for } i=1,2.
$$
Now the claim follows with $c_1=\epsilon^{-1}$, $c_2=\epsilon$ and $u=\epsilon^k$.
\end{proof}

Applying the above lemma and noting that all generators of $\mathfrak{q}$ are of the form $\pm m\epsilon^n$, where $m$ is any fixed generator, $\epsilon$ is the fundamental unit and $n$ runs over the integers, 
we may write $\Psi(\mathfrak{q})$ in the form
$$
\Psi(\mathfrak{q})= 2\sum\limits_{n\in \mathbb{Z}} f\left(\frac{\sigma_1\left(m\epsilon^n\right)}{\sqrt{N}}\right)f\left(\frac{\sigma_2\left(m\epsilon^n\right)}{\sqrt{N}}\right)= 
2\sum\limits_{n\in \mathbb{Z}} f\left(\frac{\sigma_1(m)\epsilon^{n}}{\sqrt{N}}\right)f\left(\frac{\sigma_2(m)\epsilon^{-n}}{\sqrt{N}}\right),
$$
where $(m)=\mathfrak{q}$ and $\sigma_{1,2}(m)\asymp \sqrt{\mathcal{N}(\mathfrak{q})}$. This gives immediately the lower bounds in \eqref{Psisizeobs} since 
$$
\Psi(\mathfrak{q}),\tilde{\Psi}(\mathfrak{q})\gg f\left(\frac{\sigma_1(m)}{\sqrt{N}}\right)f\left(\frac{\sigma_2(m)}{\sqrt{N}}\right)\gg 1 \quad \mbox{ if }
N\le \mathcal{N}(\mathfrak{q})\le 2N.
$$
The upper bound for $\Psi(\mathfrak{q})$ follows from
$$
\sum\limits_{n\in \mathbb{Z}} f\left(\frac{\sigma_1(m)\epsilon^{n}}{\sqrt{N}}\right)f\left(\frac{\sigma_2(m)\epsilon^{-n}}{\sqrt{N}}\right) \ll 
\sum\limits_{n=0}^{\infty} \exp\left(-\pi \mathcal{D}\mathcal{C}\cdot \frac{\mathcal{N}(\mathfrak{q})}{N}\cdot \epsilon^{2n}\right)\ll \exp\left(-\pi \mathcal{D}\mathcal{C}\cdot \frac{\mathcal{N}(\mathfrak{q})}{N}\right)\cdot \log N
$$
and the upper bound for $\tilde{\Psi}(\mathfrak{q})$ from
\begin{equation*}
\begin{split}
& \frac{N}{\mathcal{N}(\mathfrak{q})}\cdot \sum\limits_{n\in \mathbb{Z}} f\left(\frac{\sigma_1(m)\epsilon^{n}}{\sqrt{N}}\right)f\left(\frac{\sigma_2(m)\epsilon^{-n}}{\sqrt{N}}\right)\\ \ll &\frac{N}{\mathcal{N}(\mathfrak{q})} \sum\limits_{n=0}^{\infty}  \exp\left(-\pi \mathcal{D}\mathcal{C}\cdot \frac{\mathcal{N}(\mathfrak{q})}{N}\cdot \epsilon^{2n}\right) \cdot \min\left\{1,\left|\frac{\mathcal{N}(\mathfrak{q})}{N}\cdot \epsilon^{-2n}\right|^C\right\}
\\ \ll & \exp\left(-\pi \mathcal{D}\mathcal{C}\cdot \frac{\mathcal{N}(\mathfrak{q})}{N}\right)\cdot \log N
\end{split}
\end{equation*}
for a suitable constant $\mathcal{D}>0$. Here we use $\mathcal{N}(\mathfrak{q})\ge 1$ and \eqref{fdef}. 

At this point, we apply Lemma \ref{sigmasize} to establish a related result which will be needed only in subsection \ref{almostall}.

\begin{Lemma} \label{supplement}
For $q\in \mathcal{O}$, let
$$
\tilde{\mathcal{N}}(q)=\sigma_1(q)^2+\sigma_2(q)^2.
$$
Then we have 
$$
\sum\limits_{\substack{q\in \mathcal{O}\\ (q)=\mathfrak{a}}} \frac{1}{\tilde{\mathcal{N}}(q)} \ll \frac{1}{\mathcal{N}(\mathfrak{a})}
$$
for any non-zero (principal) ideal $\mathfrak{a}\in \mathcal{I}$.
\end{Lemma}

\begin{proof} As above, we may write 
$$
\sum\limits_{\substack{q\in \mathcal{O}\\ (q)=\mathfrak{a}}} \frac{1}{\tilde{\mathcal{N}}(q)}=\sum\limits_{n\in \mathbb{Z}}
 \frac{1}{\sigma_1\left(a\epsilon^n\right)^2+\sigma_2\left(a\epsilon^n\right)^2},
 $$
 where $(a)=\mathfrak{a}$ with $\sigma_{1,2}(a)\asymp \sqrt{\mathcal{N}(\mathfrak{a})}$ and $\epsilon$ is the fundamental unit. It follows that 
 $$
\sum\limits_{\substack{q\in \mathcal{O}\\ (q)=\mathfrak{a}}} \frac{1}{\tilde{\mathcal{N}}(q)}=\sum\limits_{n\in \mathbb{Z}}
 \frac{1}{\sigma_1\left(a\right)^2\epsilon^{2n}+\sigma_2\left(a\right)^2\epsilon^{-2n}}\ll \sum\limits_{n=0}^{\infty}
 \frac{1}{\mathcal{N}(\mathfrak{a})\epsilon^{2n}} \ll \frac{1}{\mathcal{N}(\mathfrak{a})},
 $$
 which completes the proof. 
\end{proof} 

Our bound \eqref{Fsize} follows from the exponential decay of the Gaussian and the fact that
$\Lambda(q;x_1,x_2)$, defined in \eqref{lattice},
is a shifted lattice in $\mathbb{R}^2$ with a fundamental parallelogram of area
equal to $1/\mathcal{N}(\mathfrak{q})$ and side lengths $\asymp 1/\sqrt{\mathcal{N}(\mathfrak{q})}$. To see the latter, we employ again Lemma \ref{sigmasize}: The set $\Lambda(q;x_1,x_2)$ remains the same if we replace
$q$ by any other generator of the ideal $\mathfrak{q}=(q)$. Using Lemma \ref{sigmasize}, we may choose $q$ in such a way that 
$\sigma_{1,2}(q)\asymp \sqrt{\mathcal{N}(\mathfrak{q})}$. Then
$$
\Lambda(q;x_1,x_2)=(x_1,x_2)-\left\{u{\bf a}+v{\bf b} : (u,v)\in \mathbb{Z}^2\right\}
$$
with 
$$
{\bf a}:=\left(\frac{1}{\sigma_1(q)},\frac{\sqrt{d}}{\sigma_2(q)}\right)
\quad \mbox{ and } \quad {\bf b}:=\left(\frac{1}{\sigma_1(q)},-\frac{\sqrt{d}}{\sigma_2(q)}\right),
$$ 
which are vectors of lengths $\asymp 1/\sqrt{\mathcal{N}(\mathfrak{q})}$ which span a parallelogram of area equal to $1/\mathcal{N}(\mathfrak{q})$,
as claimed. 

\section{Poisson summation}
In this section, we shall transform $F(\mathfrak{q})$, defined in \eqref{Fdef}, using the 2-dimensional Poisson summation formula, given below. 

\begin{Lemma} \label{Poisson}
Suppose that $f\in L^1(\mathbb{R}^n)$. Let $\hat{f}$ be the Fourier transform
of this function, defined as
$$
\hat{f}(y):=\int\limits_{\mathbb{R}^n} f(x)e(-x\cdot y) dx
$$
for $y\in \mathbb{R}^n$. Suppose that $\hat{f}\in L^1(\mathbb{R}^n)$ and
\begin{align*}
|f(x)|+|\hat{f}(x)|\ll  (1+|x|)^{-(n+\varepsilon)}
\end{align*}
for some $\varepsilon>0$ and all $x\in \mathbb{R}^n$, where $|x|$ is the Euclidean norm of $x$. Then for all $z\in \mathbb{R}^n$ we have 
\begin{align*}
\sum_{m\in\mathbb{Z}^n}\hat{f}(m)e(m\cdot z)=\sum_{m\in\mathbb{Z}^n}f(m+z),
\end{align*}
where the series on the left-hand and right-hand sides are absolutely convergent, respectively. 
In particular, 
$$
\sum\limits_{m\in\mathbb{Z}^n}\hat{f}(m)=\sum\limits_{m\in\mathbb{Z}^n}f(m).
$$
\end{Lemma}

\begin{proof} See, for example, \cite{Bump}.
\end{proof}

Using Lemma \ref{Poisson}, we shall establish the following.

\begin{Lemma} \label{poissontransform}
We have
\begin{equation} \label{FPoisson}
F(\mathfrak{q})=\frac{\mathcal{N}(\mathfrak{q})}{N}\cdot \frac{\delta^2}{2\sqrt{d}}\cdot \sum\limits_{p\in\mathcal{O}}e\left(\frac{\sigma_2(pq)x_2-\sigma_1(pq)x_1}{2\sqrt{d}}\right) \cdot \exp\left(-\pi\cdot 
\frac{\sigma_1(pq)^2+\sigma_2(pq)^2}{4d}\cdot \frac{\delta^2}{N}\right).
\end{equation}
\end{Lemma}

\begin{proof}
Let $f:\mathbb{R}^2\rightarrow \mathbb{C}$ be defined as
\begin{align*}
    f(x,y)=\Omega_{\delta/\sqrt{N}}\left(x_1-\frac{x+y\sqrt{d}}{\sigma_1(q)}\right)\Omega_{\delta/\sqrt{N}}\left(x_2-\frac{x-y\sqrt{d}}{\sigma_2(q)}\right)
\end{align*} 
so that 
$$
\Omega_{\delta/\sqrt{N}}\left(x_1-\frac{\sigma_1(p)}{\sigma_1(q)}\right)\Omega_{\delta/\sqrt{N}}\left(x_2-\frac{\sigma_2(p)}{\sigma_2(q)}\right)=f(x+y\sqrt{d})
$$
if $p=x+y\sqrt{d}$. The exponential decay of the functions $\Omega_{\delta/\sqrt{N}}$ ensures that $f\in L^1(\mathbb{R}^n)$. 
Now we calculate the Fourier transform of $f$. Making the linear change of variables
\begin{align*}
    u=\frac{\sqrt{N}}{\delta}\cdot \left(x_1-\frac{x+y\sqrt{d}}{\sigma_1(q)}\right),\quad
    v=\frac{\sqrt{N}}{\delta}\cdot\left(x_2-\frac{x-y\sqrt{d}}{\sigma_2(q)}\right),
\end{align*}
and using the definition of $\Omega_{\Delta}$ in \eqref{Omegadef}, we obtain
\begin{equation} \label{doubleint}
\begin{split}
    \hat{f}(\alpha,\beta)=& \int\limits_{\mathbb{R}^2} f(x,y)e(-x\alpha-y\beta) dydx\\ 
    = & \frac{\mathcal{N}(\mathfrak{q})}{N}\cdot \frac{\delta^2}{2\sqrt{d}} \cdot e(-A\alpha-B\beta)\cdot
    \int\limits_{\mathbb{R}^2}\exp\left(-\pi(u^2+v^2)\right) \cdot  e\left(C(\alpha,\beta)u-D(\alpha,\beta)v\right)dvdu,
\end{split}
\end{equation}
where 
\begin{equation*} \label{ABCD}
\begin{split}
A:=& \frac{\sigma_1(q)x_1+\sigma_2(q)x_2}{2}, \\ 
B:=& \frac{\sigma_1(q)x_1-\sigma_2(q)x_2}{2\sqrt{d}}, \\
C(\alpha,\beta):= & \frac{(\beta+\sqrt{d}\alpha)\sigma_1(q)}{2\sqrt{d}} \cdot
\frac{\delta}{\sqrt{N}},\\ 
D(\alpha,\beta):= & \frac{(\beta-\sqrt{d}\alpha)\sigma_2(q)}{2\sqrt{d}}\cdot
\frac{\delta}{\sqrt{N}}.
\end{split} 
\end{equation*}
Calculating the double integral in the second line of \eqref{doubleint}, we get
\begin{align*}
\begin{split}
    \hat{f}(\alpha,\beta)= & \frac{\mathcal{N}(\mathfrak{q})}{N}\cdot \frac{\delta^2}{2\sqrt{d}}\cdot e\left(\frac{(\beta-\alpha\sqrt{d})\sigma_2{(q)}x_2-(\beta+\alpha\sqrt{d})\sigma_1{(q)}x_1}{2\sqrt{d}}\right)\times\\ & \exp\left(-\pi\cdot \frac{(\beta+\alpha\sqrt{d})^2\sigma_1(q)^2+(\beta-\alpha\sqrt{d})^2\sigma_2(q)^2}{4d}\cdot \frac{\delta^2}{N}\right).
\end{split}
\end{align*}
Clearly, $\hat{f}\in L^1(\mathbb{R}^2)$. Now the Poisson summation formula
implies the result of Lemma \eqref{poissontransform} after recalling the definition
of $F(\mathfrak{q})$ in \eqref{Fdef}.
\end{proof}

Plugging \eqref{FPoisson} into \eqref{2.2} and recalling the definition of $\tilde{\Psi}(\mathfrak{q})$ in \eqref{Psidef}, we obtain
\begin{align} \label{2.3}
    \tilde{\omega}(\mathfrak{q})
    =  \Psi(\mathfrak{q})\cdot \frac{\delta^2}{2\sqrt{d}}\cdot \sum\limits_{p\in\mathcal{O}}e\left(\frac{\sigma_2(pq)x_2-\sigma_1(pq)x_1}{2\sqrt{d}}\right) \cdot \exp\left(-\pi\cdot 
\frac{\sigma_1(pq)^2+\sigma_2(pq)^2}{4d}\cdot \frac{\delta^2}{N}\right). 
    \end{align}
We note that the right-hand side is still independent of the choice of the generator $q$ of $\mathfrak{q}$. Further, we observe that $\omega(\mathfrak{q})$, defined in \eqref{2.1},
equals the contribution of $p=0$ on the right-hand side of \eqref{2.3} so that 
\begin{equation} \label{difference}
\begin{split}
 \tilde{\omega}(\mathfrak{q})-\omega(\mathfrak{q})
    = & \Psi(\mathfrak{q})\cdot \frac{\delta^2}{2\sqrt{d}}\cdot \sum\limits_{p\in\mathcal{O}\setminus\{0\}}e\left(\frac{\sigma_2(pq)x_2-\sigma_1(pq)x_1}{2\sqrt{d}}\right) \times \\ & \exp\left(-\pi\cdot 
\frac{\sigma_1(pq)^2+\sigma_2(pq)^2}{4d}\cdot \frac{\delta^2}{N}\right). 
\end{split}
\end{equation}
This will be essential in establishing a non-trivial bound for the difference $\tilde{\mathcal{T}}(N)-\mathcal{T}(N)$. The next section provides 
a version of Harman's sieve for quadratic fields, which will be a key tool in what follows.

\section{Harman's Sieve for quadratic number fields} \label{Harmansieve}
Throughout this section, we shall use the following notations.
\subsection{Notations} \label{nota}
\begin{itemize}
\item We denote by $\mathbb{R}_{\ge 0}$ the set of non-negative real numbers.
\item We suppose that $\mathbb{K}\subseteq \mathbb{R}$ is a quadratic field with ring of integers $\mathcal{O}$.
\item We denote by $\mathcal{I}$ the set of all ideals of $\mathcal{O}$. 
\item We denote by $\mathbb{P}$ the set of all non-zero prime ideals of $\mathcal{O}$.
\item We set $\mathbb{P}(z)$ = $\{\mathfrak{p}\in\mathbb{P}: \mathcal{N(\mathfrak{p})}<z\}$.
\item We set \begin{align*}
\mathcal{P}(z) = \prod_{\mathfrak{p}\in\mathbb{P}(z)} \mathfrak{p}.
\end{align*}
\item We denote by $d_k(\mathfrak{a})$ the number of ways to write an ideal $\mathfrak{a}\subseteq\mathcal{O}$ as a product of $k$ ideals. In particular, 
$d_2(\mathfrak{a})=d(\mathfrak{a})$ is the number of ideal divisors of $\mathfrak{a}$. 
\item We write $(\mathfrak{a},\mathfrak{b})=1$ if the ideals $\mathfrak{a}$ and $\mathfrak{b}$ are coprime, i.e., $\mathcal{O}$ is the only divisor of both $\mathfrak{a}$ and $\mathfrak{b}$. 
\end{itemize}

\subsection{The sieve result}
In the appendix, we shall prove the following weighted version of Harman's asymptotic sieve for ideals in the ring of integers of a quadratic field (not necessarily real quadratic and not necessarily of class number 1).

\begin{Theorem}[Weighted version of Harman's asymptotic sieve for quadratic fields] \label{T4.1} Let $x\geq3$ be real and let $\omega,\Tilde{\omega}:\mathcal{I}\longrightarrow \mathbb{R}_{\ge 0}$ be two bounded functions such that, for both $w=\omega$ and $w=\Tilde{\omega}$,
  \begin{align} \label{4.1}
  \lim_{R\to \infty} \sum_{\substack{\mathfrak{a}\in \mathcal{I} \\ \mathcal{N}(\mathfrak{a})<R}} d_4(\mathfrak{a})w(\mathfrak{a})\ll X  
\end{align}
for $X\geq 1$. Suppose further one has $Y>1$, $0<\mu<1$, $0<\kappa\le 1/2$ and $M\in(x^\mu,x)$ with the following property: \\
For any sequences $(a_\mathfrak{a})_{\mathfrak{a} \in \mathcal{I}}$,$(b_\mathfrak{b})_{\mathfrak{b} \in \mathcal{I}}$ of complex numbers with $|a_\mathfrak{a}|\leq 1$ and $|b_\mathfrak{b}|\leq d(\mathfrak{b})$, one has,
\begin{align} \label{4.2}
    T_{I}:=\Bigg|\mathop{\sum\sum}\limits_{\substack{\mathfrak{a},\mathfrak{b} \in \mathcal{I}\setminus{0}\\ \mathcal{N}(\mathfrak{a})<M}} a_\mathfrak{a} (\omega(\mathfrak{ab})-\Tilde{\omega}(\mathfrak{ab}))\Bigg| \leq Y 
\end{align}
and 
\begin{align} \label{4.3}
    T_{II}:=\Bigg|\mathop{\sum\sum}\limits_{\substack{\mathfrak{a},\mathfrak{b} \in \mathcal{I}\setminus{0}\\ x^\mu<\mathcal{N}(\mathfrak{a})<x^{\mu+\kappa}}} a_\mathfrak{a} b_\mathfrak{b} (\omega(\mathfrak{ab})-\Tilde{\omega}(\mathfrak{ab}))\Bigg| \leq Y.
\end{align}
Then 
\begin{equation} \label{thediff}
    |S(\omega,x^\kappa)-S(\Tilde{\omega},x^\kappa)|\ll Y(\log(xX))^3,
\end{equation}
where
\begin{align*}
S(w,z)=\sum_{\substack{\mathfrak{a}\in \mathcal{I}\\(\mathfrak{a},\mathcal{P}(z))=1}} w(\mathfrak{a}).
\end{align*}
\end{Theorem} 

\subsection{Applying Harman's sieve to real quadratic fields}
Now we apply Theorem \ref{T4.1} with 
$$
   \mathbb{K}:=\mathbb{Q}(\sqrt{d}),\quad\mathcal{O}=\mathbb{Z}[\sqrt{d}],\quad x^{-1/2}<\delta\leq 1/2,\quad N:=\left\lceil x^{1-\varepsilon}\right\rceil, \quad X:=x, \quad \kappa:=1/2,
$$
$$
   M:=2x^{1/4}, \quad \mu:=1/4
$$
and $\omega$ and $\tilde{\omega}$ as defined in \eqref{2.1} and \eqref{2.2}. In the following, we take into account that for every given $k\in\mathbb{N}$ and $\varepsilon>0$, $d_k(\mathfrak{q})\leq {\mathcal{N}(\mathfrak{q})}^{\varepsilon}$ if $\mathcal{N}(\mathfrak{q})$ is large enough. We shall obtain non-trivial estimates for the bilinear sums in \eqref{4.2} and \eqref{4.3} for the above choices of variables. According to the usual terminology, the sum in \eqref{4.2} is called type I and that in \eqref{4.3} type II sum. We observe that under the above conditions the sums $\tilde{\mathcal{T}}(N)$ and $\mathcal{T}(N)$ defined in \eqref{2.4} and \eqref{2.5} satisfy
\begin{align} \label{5.1}
  \mathcal{T}(N)=S(\omega,x^\kappa)+O\left(x^{1/2}\log x\right) 
\end{align}
and 
\begin{align} \label{5.2}
    \tilde{\mathcal{T}}(N)=S(\Tilde{\omega},x^\kappa)+O\left(x^{1/2}\log x\right).
\end{align}
This is because
\begin{equation*}
\begin{split}
S(w,x^{\kappa})= S(w,x^{1/2})= & \sum_{\substack{\mathfrak{a}\in \mathcal{I}\\ (\mathfrak{a},\mathcal{P}(z))=1\\ \mathcal{N}(\mathfrak{a})\le x}} w(\mathfrak{a}) +O\left(x^{-100}\right)\\ 
= & w(\mathcal{O})+\sum_{\substack{\mathfrak{p}\in \mathbb{P}\\ x^{1/2}<\mathcal{N}(\mathfrak{p})\le x}} w(\mathfrak{p})+O\left(x^{-100}\right) 
\\ = &
\sum_{\substack{\mathfrak{p}\in \mathbb{P}\\ \mathcal{N}(\mathfrak{p})\le x}} w(\mathfrak{p})+O\left(x^{1/2}\log x\right)\\ = &
\sum_{\mathfrak{p}\in \mathbb{P}} w(\mathfrak{p})+O\left(x^{1/2}\log x\right)
\end{split}
\end{equation*}
for $w=\omega,\tilde{\omega}$, where we use \eqref{omegabounds}
and $\delta\le 1$. 
In the next sections, we will deal with the Type I and Type II sums.

\section{Treatment of the type II sum}
\subsection{Initial transformations} \label{It}
Plugging \eqref{difference} into \eqref{4.3}, we obtain
\begin{equation*} \begin{split}
T_{II}=&\bigg|\mathop{\sum\sum}_{\substack{\mathfrak{m}, \mathfrak{n} \in \mathcal{I}\setminus{0}\\ x^\mu<\mathcal{N}(\mathfrak{m})<x^{\mu+\kappa}}} a_{\mathfrak{m}} b_{\mathfrak{n}} (\omega(\mathfrak{mn})-\Tilde{\omega}(\mathfrak{mn})\bigg|\\
= & \frac{\delta^2}{2\sqrt{d}}\cdot \bigg|\mathop{\sum\sum}_{\substack{\mathfrak{m},\mathfrak{n} \in \mathcal{I}\setminus{0}\\ x^\mu<\mathcal{N}(m)<x^{\mu+\kappa}}} a_{\mathfrak{m}} b_{\mathfrak{n}} 
\Psi(\mathfrak{mn})\cdot \sum\limits_{p\in \mathcal{O}\setminus\{0\}} E(pk)\mathcal{E}(pk)\bigg|,
\end{split}
\end{equation*}
where $k$ is {\it any} generator of the ideal $\mathfrak{mn}$,
and we set 
\begin{equation*}
E(l):= 
\exp\left(-\pi\cdot \frac{\sigma_1(l)^2+\sigma_2(l)^2}{4d}\cdot \frac{\delta^2}{N}\right)
\end{equation*}
and 
\begin{equation*}
\mathcal{E}(l):=e\left(\frac{\sigma_2(l)x_2-\sigma_1(l)x_1}{2\sqrt{d}}\right).
\end{equation*}
We also use the notation
$$
G(k):=f\left(\frac{\sigma_1(k)}{\sqrt{N}}\right)f\left(\frac{\sigma_2(k)}{\sqrt{N}}\right)
$$
below, where $f$ is defined as in \eqref{fdef}. We remind the reader that we aim to achieve a bound of size $O(\delta^2 N^{1-\varepsilon})$ for $T_{II}$. Since we allow $\delta$ to be as small as $N^{-\nu+15\varepsilon}$ in Theorem \ref{difftheorem} (with $\nu=1/8-\varepsilon$ in the best case), we need to push our final bound for $T_{II}$ to $O(N^{1-2\nu+29\varepsilon})$ (which is $O(N^{3/4+31\varepsilon})$ in the best case). 

We shall choose a generator $m$ for each of the ideals 
$\mathfrak{m}$ in such a 
way that $|\sigma_i(m)|\asymp \sqrt{\mathcal{N}(m)}$ for $i=1,2$, which is possible due to Lemma \ref{sigmasize}. Moreover, using the definition of 
$\Psi(\mathfrak{mn})$ and the independence of the sum over $p$ from the choice of $k$ as generator of $\mathfrak{mn}$, we write, for $\mathfrak{m},\mathfrak{n}$ and $m$ fixed,
\begin{equation}
\begin{split}
\Psi(\mathfrak{mn})\cdot \sum\limits_{p\in \mathcal{O}\setminus\{0\}} E(pk)\mathcal{E}(pk) = & \sum\limits_{\substack{k\in \mathcal{O}\\ (k)=\mathfrak{mn}}} 
G(k) \sum\limits_{p\in \mathcal{O}\setminus\{0\}} E(pk)\mathcal{E}(pk)\\
= & 
\sum\limits_{\substack{n\in \mathcal{O}\\ (n) = \mathfrak{n}}} 
\sum\limits_{p\in \mathcal{O}\setminus\{0\}} G(mn)E(pmn)\mathcal{E}(pmn)
\end{split}
\end{equation}
on setting $n:=k/m$. 
Unfolding the sum
$$
\sum\limits_{\mathfrak{n}\in I\setminus\{0\}} \sum\limits_{\substack{n\in
\mathcal{O}\\ (n)=\mathfrak{n}}},
$$
we therefore obtain
\begin{equation*}
T_{II}=\frac{\delta^2}{2\sqrt{d}}\cdot \bigg|\sum\limits_{
\substack{m\in \mathcal{R}\\ x^\mu<\mathcal{N}(m)<x^{\mu+\kappa}}}
\sum\limits_{n\in \mathcal{O}\setminus\{0\}} a_{m} b_{n} \sum\limits_{p\in \mathcal{O}\setminus\{0\}} E(p,mn)\mathcal{E}(pmn)\bigg|,
\end{equation*}
where $\mathcal{R}$ is a maximal system of mutually non-associate elements $m$ of
$\mathcal{O}$ satisfying $|\sigma_i(m)|\asymp \sqrt{\mathcal{N}(m)}$ for $i=1,2$, and 
$$
a_m:=a_{(m)}, \quad b_n:=b_{(n)}
$$
and 
\begin{equation} \label{tildeE}
\begin{split}
E(p,k):= & 
f\left(\frac{\sigma_1(k)}{\sqrt{N}}\right)f\left(\frac{\sigma_2(k)}{\sqrt{N}}\right)\exp\left(-\pi\cdot \frac{\sigma_1(pk)^2+\sigma_2(pk)^2}{4d}\cdot \frac{\delta^2}{N}\right).
\end{split}
\end{equation}

\subsection{Cutting off summations} \label{Cs}
For convenience, we would like to cut off the summation over $n$ and $p$ at appropriate points so that we are left with finite sums only. Taking  
$|\sigma_i(m)|\asymp \sqrt{\mathcal{N}(m)}$ for $i=1,2$ and $N:=\left\lceil x^{1-\varepsilon}\right\rceil$
into account, the weight function 
$$
G(mn)=f\left(\frac{\sigma_1(mn)}{\sqrt{N}}\right)f\left(\frac{\sigma_2(mn)}{\sqrt{N}}\right)
$$
becomes negligible if 
$$
\max\{|\sigma_1(n)|,|\sigma_2(n)|\}> \sqrt{\frac{x}{\mathcal{N}(m)}}.
$$
Therefore, it suffices to take only those $n$'s into consideration for which 
\begin{equation} \label{sigma1sigma2}
\max\{|\sigma_1(n)|,|\sigma_2(n)|\}\le \sqrt{\frac{x}{\mathcal{N}(m)}}.
\end{equation}
Moreover, we may discard all $p$'s for which
\begin{align} \label{pdiscard} 
    \frac{\sigma_1(pmn)^2+\sigma_2(pmn)^2}{4d}\cdot \frac{\delta^2}{N}\gg x^{\varepsilon}.
\end{align}
Now comes the point where we make use of the exponent $\mathcal{C}$. If
\begin{equation} \label{it}
\sigma_1(mn)^2\le x^{-\varepsilon}N \quad \mbox{ or } \quad \sigma_2(mn)^2\le x^{-\varepsilon}N,
\end{equation}
then 
$$
G(mn)\ll x^{-2\mathcal{C}\varepsilon}
$$
using \eqref{useful}.
Now we may choose 
$$
\mathcal{C}:=\left\lceil \frac{100}{\varepsilon} \right\rceil
$$ 
so that 
$$
E(mn)\ll x^{-200}
$$
if $mn$ satisfies \eqref{it}.
The contribution of these $mn$ becomes negligibly small so that we can assume that
$$
\sigma_1(mn)^2> x^{-\varepsilon}N \quad \mbox{ and } \quad \sigma_2(mn)^2> x^{-\varepsilon}N
$$
in which case inequality \eqref{pdiscard} holds if  
$$
|\sigma_1(p)|> x^{\varepsilon}\delta^{-1} \quad \mbox{ or } \quad  
|\sigma_2(p)|> x^{\varepsilon}\delta^{-1}.
$$
Hence, it suffices to consider $p$'s such that
$$
|\sigma_{1,2}(p)|\le x^{\varepsilon}\delta^{-1}.
$$

We deduce that
\begin{equation*}
\begin{split}
T_{II}=& \frac{\delta^2}{2\sqrt{d}}\cdot \bigg|\sum\limits_{
\substack{m\in \mathcal{R}\\ x^\mu<\mathcal{N}(m)<x^{\mu+\kappa}}}
\sum\limits_{\substack{n\in \mathcal{O}\setminus\{0\}\\ |\sigma_{1,2}(n)|\le (x/\mathcal{N}(m))^{1/2}}}
 a_{m} b_{n}\times\\ &  \sum\limits_{\substack{p\in \mathcal{O}\setminus\{0\}\\
 |\sigma_{1,2}(p)|\le x^{\varepsilon}\delta^{-1}}} E(p,mn)\mathcal{E}(pmn)\bigg|+O\left(x^{-100}\right).
\end{split}
\end{equation*}
Moreover, we divide the $m$-sum into $O(\log x)$ subsums
$$
\sum\limits_{K\le \mathcal{N}(m)\le 2K} 
$$ 
over dyadic intervals, getting
\begin{equation} \label{TKpre}
\begin{split}
T_{II}\ll & (\log x) \delta^2\sup\limits_{x^{\mu}\le K\le x^{\mu+\kappa}} |\Sigma_K|+O\left(x^{-100}\right),
\end{split}
\end{equation}
where
$$
\Sigma_K:=\sum\limits_{\substack{m\in \mathcal{R}\\K\le \mathcal{N}(m)\le 2K}} \sum\limits_{\substack{n\in \mathcal{O}\setminus\{0\}\\ |\sigma_{1,2}(n)|\le (x/K)^{1/2}}} a_{m} b_{n}\mathop{\sum}\limits_{\substack{p\in \mathcal{O}\setminus\{0\}\\|\sigma_{1,2}(p)|\le x^{\varepsilon}\delta^{-1}}}E(p,mn)\mathcal{E}(pmn)
$$
with $a_{m}:=0$ if $\mathcal{N}(m)>x^{\mu+\kappa}$.
Here we note, again looking at $G(mn)$, that the contribution of $n$'s with $(x/\mathcal{N}(m))^{1/2}<|\sigma_1(n)|\le (x/K)^{1/2}$ or 
$(x/\mathcal{N}(m))^{1/2}<|\sigma_2(n)|\le (x/K)^{1/2}$ is negligible. 

Now we use the definition of $f$ in \eqref{fdef} and expand the $\mathcal{C}$-th powers implicit in the definition of $E(p,k)$ in \eqref{tildeE}. We are led to a linear combination of sums of the form
$$
\Sigma_{K,j}:=\sum\limits_{\substack{m\in \mathcal{R}\\K\le \mathcal{N}(m)\le 2K}} \sum\limits_{\substack{n\in \mathcal{O}\setminus\{0\}\\ |\sigma_{1,2}(n)|\le (x/K)^{1/2}}} a_{m} b_{n}\mathop{\sum}\limits_{\substack{p\in \mathcal{O}\setminus\{0\}\\|\sigma_{1,2}(p)|\le x^{\varepsilon}\delta^{-1}}} E_{j}(p,mn)\mathcal{E}(pmn),
$$
where $j:=(j_1,j_2)\in \mathbb{N}^2$ and 
$$
E_j(p,k):= 
\exp\left(-\pi\cdot \frac{\sigma_1(pk)^2+\sigma_2(pk)^2}{4d}\cdot \frac{\delta^2}{N}\right)\cdot \exp\left(-\pi\cdot \frac{j_1\sigma_1(k)^2+j_2\sigma_2(k)^2}{N}\right)
$$
with $j_1$ and $j_2$ bounded by $2\mathcal{C}$. Hence, \eqref{TKpre} turns into
 \begin{equation} \label{TKturn}
\begin{split}
T_{II}\ll & (\log x) \delta^2\sup\limits_{x^{\mu}\le K\le x^{\mu+\kappa}} \sum\limits_{1\le j\le 2\mathcal{C}} |\Sigma_{K,j}|+O\left(x^{-100}\right),
\end{split}
\end{equation}
where the summation condition on $j$ means that $1\le j_1,j_2\le 2\mathcal{C}$.
In the following, we bound $|\Sigma_{K,j}|$. 

\subsection{Removing the exponential weights} \label{removal}
Whereas in the later treatment of the type I sums, the weight $E_j(p,mn)$ will be essential for performing Poisson summation in the smooth sum over $n$, it is convenient to remove it when
dealing with the type II sums. This will be done using inverse Mellin transform. Smooth weights of a more suitable shape will be re-introduced after applying Cauchy-Schwarz in section 6.4.

We begin by writing the Gaussian as an inverse Mellin transform in the form
$$
\exp\left(-y^2\right)=\frac{1}{2\pi i}\cdot \int\limits_{c-i\infty}^{c+i\infty} |y|^{-s} \cdot \frac{\Gamma(s/2)}{2} ds
$$
for all $y\in \mathbb{R}$, where $c$ is any positive number. This implies
$$
\exp\left(-\pi\cdot \frac{\sigma_1(pmn)^2}{4d}\cdot \frac{\delta^2}{N}\right)=
\frac{1}{4\pi i}\cdot \int\limits_{c-i\infty}^{c+i\infty} \left(\frac{\sqrt{\pi}\delta
|\sigma_1(pmn)|}{\sqrt{4dN}}\right)^{-s_1} \Gamma\left(\frac{s_1}{2}\right) ds_1,
$$
$$
\exp\left(-\pi\cdot \frac{\sigma_2(pmn)^2}{4d}\cdot \frac{\delta^2}{N}\right)=
\frac{1}{4\pi i}\cdot \int\limits_{c-i\infty}^{c+i\infty} \left(\frac{\sqrt{\pi}\delta
|\sigma_2(pmn)|}{\sqrt{4dN}}\right)^{-s_2} \Gamma\left(\frac{s_2}{2}\right) ds_2,
$$
$$
\exp\left(-\pi\cdot \frac{j_1\sigma_1(mn)^2}{N}\right)=
\frac{1}{4\pi i}\cdot \int\limits_{c-i\infty}^{c+i\infty} \left(\frac{\sqrt{j_1\pi}
|\sigma_1(mn)|}{\sqrt{N}}\right)^{-s_3} \Gamma\left(\frac{s_3}{2}\right) ds_3
$$
and 
$$
\exp\left(-\pi\cdot \frac{j_2\sigma_2(mn)^2}{N}\right)=
\frac{1}{4\pi i}\cdot \int\limits_{c-i\infty}^{c+i\infty} \left(\frac{\sqrt{j_2\pi}
|\sigma_2(mn)|}{\sqrt{N}}\right)^{-s_4} \Gamma\left(\frac{s_4}{2}\right) ds_4.
$$
Write ${\bf s}:=(s_1,s_2,s_3,s_4)$ and $d{\bf s}:=
ds_4ds_3ds_2ds_1$. Then it follows that
\begin{equation} \label{sigmaKs}
\Sigma_{K,j}=  \frac{1}{(4\pi i)^4}\cdot \int\limits_{c-i\infty}^{c+i\infty} 
\cdots \int\limits_{c-i\infty}^{c+i\infty} j_1^{-s_3/2}j_2^{-s_4/2} 
\Phi({\bf s})\prod\limits_{i=1}^4 \Gamma\left(\frac{s_i}{2}\right)
\Sigma_{K}({\bf s})d{\bf s},
\end{equation}
where 
\begin{equation*}
\Phi({\bf s}):=(N/\pi)^{(s_1+s_2+s_3+s_4)/2}(4d)^{(s_1+s_2)/2} \delta^{-(s_1+s_2)}
\end{equation*}
and 
$$
\Sigma_K({\bf s}):=\sum\limits_{\substack{m\in \mathcal{R}\\K\le \mathcal{N}(m)\le 2K}} \sum\limits_{\substack{n\in \mathcal{O}\setminus\{0\}\\ |\sigma_{1,2}(n)|\le (x/K)^{1/2}}} a_{m}({\bf s}) b_{n}({\bf s})\sum\limits_{\substack{p\in \mathcal{O}\setminus\{0\}\\|\sigma_{1,2}(p)|\le x^{\varepsilon}\delta^{-1}}} c_p(s_1,s_2)\mathcal{E}(pmn)
$$
with 
\begin{equation*} 
\begin{split}
a_m({\bf s}):=& a_m|\sigma_1(m)|^{-s_1-s_3}|\sigma_2(m)|^{-s_2-s_4},\\
b_n({\bf s}):=& b_m|\sigma_1(n)|^{-s_1-s_3}|\sigma_2(n)|^{-s_2-s_4},\\
c_p(s_1,s_2):=& |\sigma_1(p)|^{-s_1}|\sigma_2(p)|^{-s_2}.
\end{split}
\end{equation*}
We set
\begin{equation} \label{cset}
c:=\frac{1}{\log x}.
\end{equation}
Then, if 
$$
N^{-1}\le \delta\le 1,
$$
we have 
\begin{equation} \label{Phibound}
\Phi({\bf s})=O(1)
\end{equation}
and 
$$
a_m({\bf s})\ll |a_m|\le 1, \quad b_n({\bf s})\ll |b_n|\le d((n)), \quad c_p(s_1,s_2)=O(1)
$$
for all ${\bf s}$ with $\Re(s_i)=c$ for $i=1,...,4$ and $m,n,p$ in the relevant summation ranges. 

\subsection{Applying Cauchy-Schwarz} \label{CauSch}
Next we bound $\Sigma_K({\bf s})$. We first re-arrange summations and use the triangle inequality and 
the bounds $c_p(s_1,s_2)=O(1)$ and $a_m=O(1)$ to get
$$
\Sigma_K({\bf s})\ll \sum\limits_{\substack{p\in \mathcal{O}\setminus\{0\}\\|\sigma_{1,2}(p)|\le x^{\varepsilon}\delta^{-1}}}   \sum\limits_{\substack{m\in \mathcal{R}\\K\le \mathcal{N}(m)\le 2K}} \left|\sum\limits_{\substack{n\in \mathcal{O}\setminus\{0\}\\ |\sigma_{1,2}(n)|\le (x/K)^{1/2}}}b_{n}({\bf s}) \mathcal{E}(pmn)\right|.
$$
Now we apply the Cauchy-Schwarz inequality and the definition of $\mathcal{R}$ to get
\begin{equation*}
\left|\Sigma_K({\bf s})\right|^2\ll  x^{2\varepsilon} \delta^{-2}K  
\sum\limits_{\substack{p\in \mathcal{O}\setminus\{0\}\\|\sigma_{1,2}(p)|\le x^{\varepsilon}\delta^{-1}}}   \sum\limits_{\substack{m\in \mathcal{R}\\K\le \mathcal{N}(m)\le 2K}} \left|\sum\limits_{\substack{n\in \mathcal{O}\setminus\{0\}\\ |\sigma_{1,2}(n)|\le (x/K)^{1/2}}}b_{n}({\bf s}) \mathcal{E}(pmn)\right|^2.
\end{equation*}
Writing $k=pm$ and using the definition of $\mathcal{R}$, we deduce that
\begin{equation*}
\left|\Sigma_K({\bf s})\right|^2\ll  x^{2\varepsilon} \delta^{-2}K  
\sum\limits_{\substack{k\in \mathcal{O}\setminus\{0\}\\|\sigma_{1,2}(k)|\le cx^{\varepsilon}\delta^{-1}K^{1/2}}}   
\Bigg(\sum\limits_{\substack{p,m\in \mathcal{O}\\ pm=k\\ m\in \mathcal{R}}} 1\Bigg)
 \left|\sum\limits_{\substack{n\in \mathcal{O}\setminus\{0\}\\ |\sigma_{1,2}(n)|\le (x/K)^{1/2}}}b_{n}({\bf s}) \mathcal{E}(kn)\right|^2
\end{equation*}
for some constant $c>0$. Since the number of ideal divisors of $k$ is bounded by $O(\mathcal{N}(k)^{\varepsilon})$ and $m$ runs over mutually non-associate elements of $\mathcal{O}$, we have
$$
\sum\limits_{\substack{p,m\in \mathcal{O}\\ pm=k\\ m\in \mathcal{R}}} 1\ll x^{\varepsilon}.
$$
Furthermore, we introduce a smooth weight to extend the summation over $k$ to all integers, obtaining
\begin{equation*}
\left|\Sigma_K({\bf s})\right|^2\ll  x^{3\varepsilon} \delta^{-2}K  
\sum\limits_{k\in \mathcal{O}}  \exp\left(-\pi \cdot \frac{\sigma_1(k)^2+\sigma_2(k)^2}{x^{2\varepsilon}\delta^{-2}K}\right)
 \left|\sum\limits_{\substack{n\in \mathcal{O}\setminus\{0\}\\ |\sigma_{1,2}(n)|\le (x/K)^{1/2}}}b_{n}({\bf s}) \mathcal{E}(kn)\right|^2.
\end{equation*}
We expand the square, use the fact that $\overline{\mathcal{E}(l)}=\mathcal{E}(-l)$ and move in the summation over $k$ to deduce that
\begin{equation*}
\begin{split}
\left|\Sigma_K({\bf s})\right|^2\ll & x^{3\varepsilon} \delta^{-2}K  
\sum\limits_{\substack{n_1,n_2\in \mathcal{O}\setminus\{0\}\\ |\sigma_{1,2}(n_1)|\le (x/K)^{1/2}\\ |\sigma_{1,2}(n_2)|\le (x/K)^{1/2}}}b_{n_1}({\bf s})\overline{b_{n_2}({\bf s})}\times\\ & 
\sum\limits_{k\in \mathcal{O}}  \exp\left(-\pi \cdot \frac{\sigma_1(k)^2+\sigma_2(k)^2}{x^{2\varepsilon}\delta^{-2}K}\right)
 \mathcal{E}\left(k(n_1-n_2)\right).
 \end{split}
\end{equation*}
Using $b_{n_i}({\bf s})\ll d((n_i))\ll x^{\varepsilon}$, and writing $n=n_1-n_2$, it follows that
\begin{equation*}
\begin{split}
\left|\Sigma_K({\bf s})\right|^2\ll & x^{5\varepsilon} \delta^{-2}K  
\sum\limits_{\substack{n\in \mathcal{O}\\ |\sigma_{1,2}(n)|\le 2(x/K)^{1/2}}} 
\Big(\sum\limits_{\substack{n_1-n_2=n\\ |\sigma_{1,2}(n_1)|\le (x/K)^{1/2}\\ |\sigma_{1,2}(n_2)|\le (x/K)^{1/2}}}
1 \Big) \times\\ &
 \left| \sum\limits_{k\in \mathcal{O}}  \exp\left(-\pi \cdot \frac{\sigma_1(k)^2+\sigma_2(k)^2}{x^{2\varepsilon}\delta^{-2}K}\right)
 \mathcal{E}\left(kn\right)\right|.
 \end{split}
\end{equation*}
Clearly,
$$
\sum\limits_{\substack{n_1-n_2=n\\ |\sigma_{1,2}(n_1)|\le (x/K)^{1/2}\\ |\sigma_{1,2}(n_2)|\le (x/K)^{1/2}}} 1
\ll xK^{-1},
$$
giving
\begin{equation*}
\left|\Sigma_K({\bf s})\right|^2\ll  x^{1+5\varepsilon} \delta^{-2}  
\sum\limits_{\substack{n\in \mathcal{O}\\ |\sigma_{1,2}(n)|\le 2(x/K)^{1/2}}} 
\left|
\sum\limits_{k\in \mathcal{O}}  \exp\left(-\pi \cdot \frac{\sigma_1(k)^2+\sigma_2(k)^2}{x^{2\varepsilon}\delta^{-2}K}\right)
 \mathcal{E}\left(kn\right)\right|.
\end{equation*}

\subsection{Applying Poisson summation}
Now we use again the Poisson summation formula to transform the sum over 
$k$ above. For $(u,v)\in \mathbb{R}^2$ set 
$$
f(u,v):=\exp\left(-\pi \cdot \frac{2(u^2+v^2d)}{x^{2\varepsilon}\delta^{-2}K}\right) \cdot 
 e\left(\frac{(\sigma_2(n)x_2-\sigma_1(n)x_1)u-(\sigma_2(n)x_2+\sigma_1(n)x_1)v\sqrt{d}}{2\sqrt{d}}\right)
$$
so that
$$
\exp\left(-\pi \cdot \frac{\sigma_1(k)^2+\sigma_2(k)^2}{x^{2\varepsilon}\delta^{-2}K}\right)
 \mathcal{E}\left(kn\right)=f(u,v)
$$
if $k=u+v\sqrt{d}$ with $(u,v)\in \mathbb{Z}^2$. Clearly, $f\in L^1(\mathbb{R}^2)$. We calculate the Fourier transform of $f$ to be
\begin{equation*}
\begin{split}
& \hat{f}(\alpha,\beta)=\frac{1}{2\sqrt{d}}\cdot x^{2\varepsilon}\delta^{-2}K\times\\ &
\exp\left(-\frac{\pi}{2} \cdot x^{2\varepsilon}\delta^{-2}K\left(\left(\beta-\frac{\sigma_2(n)x_2-\sigma_1(n)x_1}{2\sqrt{d}}\right)^2+\left(\alpha-\frac{-(\sigma_2(n)x_2+\sigma_1(n)x_1)}{2}\right)^2\right)\right).
\end{split}
\end{equation*}
Obviously, this function is also in $L^1(\mathbb{R}^2)$. Hence, by Lemma \ref{Poisson}, we get
\begin{equation*}
\begin{split}
& \sum\limits_{k\in \mathcal{O}}  \exp\left(-\pi \cdot \frac{\sigma_1(k)^2+\sigma_2(k)^2}{x^{2\varepsilon}\delta^{-2}K}\right)\mathcal{E}\left(kn\right)=
\frac{1}{2\sqrt{d}}\cdot x^{2\varepsilon}\delta^{-2}K\times\\
&\sum\limits_{(\alpha,\beta)\in \mathbb{Z}^2} 
\exp\left(-\frac{\pi}{2} \cdot x^{2\varepsilon}\delta^{-2}K\left(\left(\beta-\frac{\sigma_2(n)x_2-\sigma_1(n)x_1}{2\sqrt{d}}\right)^2+\left(\alpha-\frac{-(\sigma_2(n)x_2+\sigma_1(n)x_1)}{2}\right)^2\right)\right),
\end{split}
\end{equation*}
which is negligibly small if 
$$
\left|\left| \frac{\sigma_2(n)x_2-\sigma_1(n)x_1}{2\sqrt{d}}\right|\right| > 
\delta K^{-1/2}
$$
or 
$$
\left|\left|\frac{\sigma_2(n)x_2+\sigma_1(n)x_1}{2} \right|\right| > 
\delta K^{-1/2}
$$
and bounded by 
$
O\left(x^{2\varepsilon}\delta^{-2}K\right)
$
otherwise. Hence, we obtain
\begin{equation} \label{precount}
\begin{split}
& \left|\Sigma_K({\bf s})\right|^2\ll  x^{1+7\varepsilon} \delta^{-4}K  \times\\ &
\sum\limits_{\substack{n\in \mathcal{O}\\ |\sigma_{1,2}(n)|
\le 2(x/K)^{1/2}}}\chi_J\left(\left|\left|\frac{\sigma_2(n)x_2-\sigma_1(n)x_1)}{2\sqrt{d}}\right|\right|\right) \cdot \chi_J\left( 
\left|\left|\frac{\sigma_2(n)x_2+\sigma_1(n)x_1}{2}\right|\right|\right),
\end{split}
\end{equation}
where $\chi_J$ is the characteristic function of the interval
$$
J:=\left[-\delta K^{-1/2},\delta K^{-1/2}\right].
$$

\section{Counting problem} \label{counting}
We are now down to a counting problem reminiscent of that appearing in the treatments of the same problem in the settings of rational or Gaussian integers. However, the present counting problem has a different flavor since it is a truly 2-dimensional problem. Indeed, we will see that new ingredients are required such as results about the approximation of roots of quadratic congruences. Our approach is similar to that in the settings of rational or Gaussian primes, namely to use Diophantine approximation to replace $x_1$ and $x_2$ by elements of $\mathbb{Q}(\sqrt{d})$. 
 
\subsection{Approximating $(x_1,x_2)$} \label{approstart}
Dirichlet's approximation theorem for $\mathbb{Q}(\sqrt{d})$ gives us, in a precise sense, a set of $\theta\in \mathbb{Q}(\sqrt{d})$ such that $x_1$ and $x_2$ have a good simultaneous approximation by $\sigma_1(\theta)$ and 
$\sigma_2(\theta)$, respectively. We will work out the details in section \ref{Dirichlet}. Assume for now we have an approximation 
satisfying
\begin{equation} \label{Diriapprox}
\left|x_i- \sigma_i(\theta)\right|\le \Delta \mbox{ for }
i=1,2,
\end{equation}
where $\theta\in \mathbb{Q}(\sqrt{d})$. Then it follows from \eqref{precount} that
\begin{equation*}
\begin{split}
& \left|\Sigma_K({\bf s})\right|^2\ll  x^{1+7\varepsilon} \delta^{-4}K  \times\\ &
\sum\limits_{\substack{n\in \mathcal{O}\\ |\sigma_{1,2}(n)|
\le 2(x/K)^{1/2}}}\chi_{\mathcal{J}}\left(\left|\left|\frac{\sigma_2(n\theta)-\sigma_1(n\theta)}{2\sqrt{d}}\right|\right|\right) \cdot \chi_{\mathcal{J}}\left( 
\left|\left|\frac{\sigma_2(n\theta)+\sigma_1(n\theta)}{2}\right|\right|\right),
\end{split}
\end{equation*}
where 
\begin{equation} \label{tJ}
\mathcal{J}:=[-\tilde{\Delta},\tilde{\Delta}]
\end{equation}
with 
\begin{equation} \label{tDeltadef}
\tilde{\Delta}:=\delta K^{-1/2}+2x^{1/2}K^{-1/2}\Delta.
\end{equation}
Writing
$$
n=\lambda+\mu \sqrt{d}, \quad \theta=\sigma+\tau\sqrt{d}
$$
with $\lambda,\mu\in \mathbb{Z}$ and $\sigma,\tau\in \mathbb{Q}$,
the above turns into
\begin{equation*}
\left|\Sigma_K({\bf s})\right|^2\ll  x^{1+7\varepsilon} \delta^{-4}K  
\sum\limits_{\substack{(\lambda,\mu)\in \mathbb{Z}^2\\ |\lambda|,|\mu|
\le 2(x/K)^{1/2}}}\chi_{\mathcal{J}}\left(\left|\left|\lambda\tau+\mu\sigma\right|\right|\right) \cdot \chi_{\mathcal{J}}\left( 
\left|\left| \lambda\sigma+\mu\tau d\right|\right|\right),
\end{equation*}
which may be rewritten as
\begin{equation*}
\left|\Sigma_K({\bf s})\right|^2\ll  x^{1+7\varepsilon} \delta^{-4}K  
\sum\limits_{\substack{(\lambda,\mu)\in \mathbb{Z}^2\\ |\lambda|,|\mu|
\le 2(x/K)^{1/2}}} \chi_{\mathcal{J}\times\mathcal{J}}\left(\left|\left|\begin{pmatrix} \sigma & \tau d\\ \tau & \sigma \end{pmatrix} \begin{pmatrix} \lambda \\ \mu 
\end{pmatrix} \right|\right|\right),
\end{equation*}
where we set
$$
\left|\left| \begin{pmatrix} x \\ y \end{pmatrix} \right|\right|:=
\begin{pmatrix} ||x|| \\ ||y|| \end{pmatrix}.
$$
Hence, if 
\begin{equation} \label{thetawrite}
\theta=\frac{a}{W}+\frac{b}{W}\sqrt{d} \quad \mbox{ with } a,b\in \mathbb{Z},\ W\in \mathbb{N},
\end{equation}
then 
\begin{equation*}
\left|\Sigma_K({\bf s})\right|^2\ll  x^{1+7\varepsilon} \delta^{-4}K  
\sum\limits_{\substack{(\lambda,\mu)\in \mathbb{Z}^2\\ |\lambda|,|\mu|
\le 2(x/K)^{1/2}}} \chi_{\mathcal{J}\times\mathcal{J}}\left(\left|\left|\begin{pmatrix} a/W & b d/W\\ b/W & a/W \end{pmatrix} \begin{pmatrix} \lambda \\ \mu 
\end{pmatrix} \right|\right|\right).
\end{equation*}
Finally, we may express the above using congruences modulo $W$ as 
\begin{equation} \label{simplewrite0}
\left|\Sigma_K({\bf s})\right|^2\ll  x^{1+7\varepsilon} \delta^{-4}K  
\mathop{\sum\limits_{\substack{(\lambda,\mu)\in \mathbb{Z}^2\\ |\lambda|,|\mu|
\le U}} \sum\limits_{\substack{(\alpha,\beta)\in \mathbb{Z}^2\\ |\alpha|,|\beta|
\le V}}}_{\mbox{\bf C}} 1,
\end{equation}
where ``{\bf C}'' stands for the congruence condition
\begin{align*}
    \begin{pmatrix}
   a & bd \\
   b & a
    \end{pmatrix}
   \begin{pmatrix}
   \lambda \\
   \mu
    \end{pmatrix}\equiv
    \begin{pmatrix}
    \alpha \\
    \beta
    \end{pmatrix}  \bmod{W} 
\end{align*} 
and
\begin{equation} \label{UVdef}
U:=2x^{1/2}K^{-1/2} \quad \mbox{ and } \quad 
V:=W\tilde{\Delta}.
\end{equation}

\subsection{A congruence relation between $a$ and $b$} \label{thetasetup}
In our application of the Dirichlet approximation theorem for $\mathbb{Q}(\sqrt{d})$ in section \ref{Dirichlet}, we shall write 
$\theta$ in the form
\begin{equation} \label{thetaDiri}
\theta=\frac{u+v\sqrt{d}}{f+g\sqrt{d}}
\end{equation}
where $u,v,f,g\in \mathbb{Z}$ and
\begin{equation} \label{copri}
\mbox{gcd}(u+v\sqrt{d},f+g\sqrt{d})\approx 1 \mbox{ in } \mathcal{O}.
\end{equation}
Further, we shall make the natural choice
$$
W:=\mathcal{N}(f+g\sqrt{d})=|f^2-g^2d|
$$
so that
$$
|a+b\sqrt{d}|=|(u+v\sqrt{d})(f-g\sqrt{d})|.
$$
It follows that 
$$
|a^2-b^2d|=\mathcal{N}(a+b\sqrt{d})=\mathcal{N}(u+v\sqrt{d})W
$$
and hence 
\begin{equation} \label{congrurel}
a^2-b^2d\equiv 0 \bmod{W}.
\end{equation}
This congruence relation will be crucial in what follows. 

\subsection{Counting solutions of systems of linear congruences}\label{countsol}
Recalling \eqref{simplewrite0}, we need to detect $\alpha,\beta,\lambda,\mu$ satisfying the system of congruences
\begin{equation} \label{8.1}
\begin{cases}
    a\lambda+bd\mu & \equiv \alpha \bmod{W},\\
    b\lambda+a\mu & \equiv \beta \bmod{W}.
\end{cases}
\end{equation}
Multiplying the first congruence with $b$ and the second with $a$ and then subtracting and using \eqref{congrurel}, we get
\begin{equation} \label{8.2}
    b \alpha  \equiv a \beta  \bmod{W}.
\end{equation}
We first count pairs $(\alpha,\beta)$ lying the relevant range for which 
\eqref{8.2} is satisfied. Then we fix $(\alpha,\beta)$ and count how many pairs $(\lambda,\mu)$ in the relevant range satisfy \eqref{8.1}. After fixing $(\alpha,\beta)$, assuming that  $(\lambda_0,\mu_0)$ is one particular solution, any other solution of \eqref{8.1} is of the
form $(\lambda_0+\tilde{\lambda},\mu_0+\tilde{\mu})$, where
 \begin{equation} \label{8.1'}
\begin{cases}
    a\tilde{\lambda}+bd\tilde{\mu} & \equiv 0 \bmod{W},\\
    b\tilde{\lambda}+a\tilde{\mu} & \equiv 0 \bmod{W}.
\end{cases}
\end{equation}
We throw away the first congruence and just count the number of pairs $(\tilde{\lambda},\tilde{\mu})$ satisfying
\begin{equation} \label{8.3}
    b \tilde{\lambda} +a\tilde{\mu} \equiv 0  \bmod{W},
\end{equation}
where $(\lambda_0+\tilde{\lambda},\mu_0+\tilde{\mu})$ lies in the relevant range, i.e.
$$
|\lambda_0+\tilde{\lambda}|,|\mu_0+\tilde{\mu}|\le U.
$$
Since 
$$
|\lambda_0|,|\mu_0|\le U,
$$
it suffices to count $(\tilde{\lambda},\tilde{\mu})$ satisfying \eqref{8.3} and
$$
|\tilde{\lambda}|,|\tilde{\mu}|\le 2U.
$$
So relabeling $\tilde{\lambda}$ and $\tilde{\mu}$ as $\lambda$ and $-\mu$, respectively, we are led to counting solutions $(\alpha,\beta,\lambda,\mu)$ of the system of independent congruences
\begin{equation} \label{8.4}
\begin{cases}
    b \alpha  & \equiv  a \beta  \bmod{W}, \\
    b \lambda & \equiv  a\mu  \bmod{W},
\end{cases}
\end{equation}
subject to the conditions 
$$
|\lambda|,|\mu|\le 2U \quad \mbox{ and } \quad |\alpha|,|\beta|\le V.
$$
These congruences are of the same shape, and therefore it suffices to count
the number $\Theta(X;a,b;W)$ of solutions $(A,B)$ with 
$$
|A|,|B|\le X
$$
of the single congruence
\begin{equation} \label{singlecon}
b B  \equiv  a A  \bmod{W}.
\end{equation}
In summary, we conclude from \eqref{simplewrite0} that
\begin{equation} \label{simplewrite}
\left|\Sigma_K({\bf s})\right|^2\ll  x^{1+7\varepsilon} \delta^{-4}K\cdot  
\Theta(2U;a,b;W)\cdot \Theta(V;a,b;W).
\end{equation}

\subsection{Detecting congruences using additive characters} \label{additive}
Next we bound $\Theta(X;a,b;W)$ under the condition that
\begin{equation} \label{abqcond}
a=Za', \quad b=Zb', \quad W=Z^2W', \quad 
\mbox{gcd}(b',W')=1
\end{equation}
or 
\begin{equation} \label{abqcond2}
a=Za', \quad b=Zb', \quad W=Z^2W', \quad \mbox{gcd}(a',W')=d
\end{equation}
for suitable $a',b',W',Z\in \mathbb{Z}$ with $W',Z>0$. In section \ref{Dirichlet}, we will prove that \eqref{abqcond} or \eqref{abqcond2} holds. 

We observe that under both \eqref{abqcond} and \eqref{abqcond2}, \eqref{congrurel} is equivalent to
\begin{equation} \label{congrurel'}
(a')^2\equiv (b')^2d \bmod{W'}
\end{equation} 
on taking out a factor of $Z^2$, and \eqref{singlecon} is equivalent to
\begin{equation}\label{singlecon'}
b'B\equiv a'A \bmod{W'}Z
\end{equation}
on taking out a factor of $Z$. Moreover, \eqref{singlecon'} implies the congruence
\begin{equation}\label{sc}
b'B\equiv a'A \bmod{W'}.
\end{equation}
We will use \eqref{sc} in place of \eqref{singlecon'} because we are not able to handle the extra factor $Z$ in \eqref{singlecon'} which does not occur in the modulus of the important congruence relation \eqref{congrurel'}. Further, under \eqref{abqcond}, if $\overline{b'}$ is a multiplicative inverse
of $b' \bmod{W'}$ (i.e., $b\overline{b'} \equiv 1 \bmod{W'}$), then \eqref{sc}
is equivalent to
\begin{equation}\label{sc'}
B\equiv a'\overline{b'}A \bmod{W'}.
\end{equation}

We will confine ourselves to treating the case when \eqref{abqcond} is satisfied because the other case is very similar and gives the same results. Indeed, in the case of \eqref{abqcond2}, 
we write $a''=a'/d$ and $W''=W'/d$ and reduce \eqref{congrurel'} further to
\begin{equation*} 
(a'')^2d\equiv (b')^2 \bmod{W''}.
\end{equation*}
Moreover, we deduce from \eqref{sc} that
\begin{equation*}
b'B\equiv a''Ad \bmod{W''}
\end{equation*}
and, noting gcd$(a'',W'')=1$, transform this into 
\begin{equation*}
Ad \equiv b'\overline{a''}B \bmod{W''},
\end{equation*}
where $\overline{a''}$ is a multiplicative inverse of $a''$ modulo $W''$.
The counting problem now takes the same shape as above, with the replacements
\begin{equation*}
\begin{split}
W' \rightarrow & W''\\
a' \rightarrow & b'\\
b' \rightarrow & a''\\
A \rightarrow & B\\
B \rightarrow & Ad.
\end{split}
\end{equation*}
An inspection of the method below shows that the occurrence of an extra factor of $d$ in the last line presents no problems. Thus, we assume \eqref{abqcond} to hold throughout the following. 

We detect the congruence \eqref{sc'} using additive characters, writing
\begin{equation} \label{addchar}
\begin{split}
& \Theta(X;a,b;W) \le \sum\limits_{\substack{|A|,|B|\le X\\ B \equiv a'\overline{b'}A \bmod{W'}}} 1 \\
= & \frac{1}{W'}\sum\limits_{|A|,|B|\le X} \sum\limits_{h=0}^{W'-1} 
e\left(\frac{h(B-a'\overline{b'}A)}{W'}\right)\\
= & \frac{1}{W'} \sum\limits_{h=0}^{W'-1} 
\left(\sum\limits_{|B|\le X} e\left(\frac{hB}{W'}\right)\right)\left(
\sum\limits_{|A|\le X} e\left(-\frac{ha'\overline{b'}A}{W'}\right)\right)\\
\ll & \frac{1}{W'} \sum\limits_{h=0}^{W'-1} \min\left\{X,\left|\left|\frac{h}{W'}\right|\right|^{-1}\right\}
\min\left\{X,\left|\left|\frac{ha'\overline{b'}}{W'}\right|\right|^{-1}\right\}\\
\ll & \frac{1}{W'} \sum\limits_{0\le h \le W'/2} 
\min\left\{X,\frac{W'}{h}\right\}
\min\left\{X,\left|\left|\frac{ha'\overline{b'}}{W'}\right|\right|^{-1}\right\}\\
\ll & \frac{X}{W'} \sum\limits_{0\le h\le W'/X} \min\left\{X,\left|\left|\frac{ha'\overline{b'}}{W'}\right|\right|^{-1}\right\} +\
\sum\limits_{W'/X< h\le W'/2} \frac{1}{h} \min\left\{X,\left|\left|\frac{ha'\overline{b'}}{W'}\right|\right|^{-1}\right\}\\
\ll & (\log 2W')\sup\limits_{\substack{H\in \mathbb{R}\\ W'/X\le H\le W'}} \frac{1}{H} \sum\limits_{0\le h\le H} \min\left\{X,\left|\left|\frac{ha'\overline{b'}}{W'}\right|\right|^{-1}\right\}
\end{split}
\end{equation}
if $X\ge 1$. Now we use the following lemma, which is a standard tool in this circle of
problems.

\begin{Lemma} \label{standardtool} Suppose that $X,H\ge 1$ and $\gamma\in \mathbb{R}$
satisfies
\begin{equation} \label{Diop}
\left|\gamma-\frac{u}{r}\right| \ll r^{-2}
\end{equation}
for some $u\in \mathbb{Z}$ and $r\in \mathbb{N}$ with {\rm gcd}$(u,r)=1$, 
then
\begin{equation} \label{standest}
\sum\limits_{0\le h\le H} \min\left\{X,||h\gamma||^{-1}\right\}
\ll \left(1+ \frac{H}{r}\right)\left(X+r\right)\log 2r,
\end{equation}
where the implied constant in \eqref{standest} depends only on that in \eqref{Diop}.
\end{Lemma}

\begin{proof}
This can be proved in a similar way as \cite[Lemma 6.4.4]{Bru} by dividing the summation range into intervals of length $r$. 
\end{proof}

Applying Lemma \ref{standardtool} to the last line of \eqref{addchar}, we deduce that
\begin{equation} \label{Thetaest}
\Theta(X;a,b;W)\ll \left(\frac{X^2}{W'}+\frac{rX}{W'}+\frac{X}{r}+1\right)(\log 2r)(\log 2W'),
\end{equation} 
provided we have a Diophantine approximation of the form
\begin{equation} \label{crucialDio}
\left| \frac{a'\overline{b'}}{W'} -\frac{u}{r} \right| \ll r^{-2} \mbox{ with gcd}(u,r)=1.
\end{equation}
We shall prove in section \ref{roots} that this is the case for some $r\asymp \sqrt{W'}$, which is a crucial point in this article. Therefore, \eqref{Thetaest} 
turns into
\begin{equation} \label{Thetaestfinal}
\Theta(X;a,b;W)\ll \left(\frac{X^2}{W'}+1\right)\log^22W'
\ll \left(\frac{X^2Z^2}{W}+1\right)\log^2 W.
\end{equation} 

\subsection{Estimating the type II sum}
Plugging \eqref{Thetaestfinal} into \eqref{simplewrite} and recalling the definitions of $U$ and $V$ in \eqref{UVdef} and $\tilde{\Delta}$ in \eqref{tDeltadef}, we obtain the estimate
\begin{equation} \label{aftercount}
\begin{split}
\left|\Sigma_K({\bf s})\right|^2\ll & x^{1+8\varepsilon} \delta^{-4}K  
\left(xK^{-1}W^{-1}Z^2+1\right)\left(\left(\delta^2 +x\Delta^2\right)K^{-1}WZ^2+1\right)\\
= & x^{8\varepsilon} \delta^{-4}\left(\left(\delta^2 +x\Delta^2\right)\left(x^2K^{-1}Z^4+xWZ^2\right)+x^2W^{-1}Z^2+xK\right),
\end{split}
\end{equation}
provided that 
\begin{equation} \label{qcond}
W\ll x^{100}.
\end{equation}
Taking square-root on both sides of \eqref{aftercount}, plugging the resulting estimate into \eqref{sigmaKs}, recalling \eqref{cset} and \eqref{Phibound}, applying Stirling's formula to bound the Gamma factors, and integrating, we obtain
\begin{equation*}
\left|\Sigma_{K,j}\right|\ll 
x^{5\varepsilon} \delta^{-2}\left(\left(\delta+x^{1/2}\Delta\right)\left(xK^{-1/2}Z^2+x^{1/2}W^{1/2}Z\right)+xW^{-1/2}Z+x^{1/2}K^{1/2}\right).
\end{equation*}
Reversing the roles of $m$ and $n$ in the whole process, we get the same estimate with $x/K$ in place of $K$, i.e.,
\begin{equation*}
\left|\Sigma_{K,j}\right|\ll 
x^{5\varepsilon} \delta^{-2}\left(\left(\delta+x^{1/2}\Delta\right)\left(x^{1/2}K^{1/2}Z^2+x^{1/2}W^{1/2}Z\right)+xW^{-1/2}Z+xK^{-1/2}\right).
\end{equation*}
Taking the first estimate when $K\le x^{1/2}$ and the second one if $K\ge x^{1/2}$ and employing 
\eqref{TKturn}, we obtain
\begin{equation} \label{TKaftercount}
\begin{split}
T_{II}\ll & 
x^{6\varepsilon}\left(\left(\delta+x^{1/2}\Delta\right)\left(x^{1-\mu/2}Z^2+x^{(1+\mu+\kappa)/2}Z^2+x^{1/2}W^{1/2}Z\right)+xW^{-1/2}Z+x^{3/4}\right)
\end{split}
\end{equation}
if 
$$
\mu\le 1/2\le \mu+\kappa.
$$

\section{Roots of quadratic congruences} \label{roots}
Now we want to establish the important Diophantine approximation
\eqref{crucialDio}.  We recall that we have the congruence \eqref{congrurel'} 
which is equivalent to
\begin{equation} \label{congrurel2}
(a'\overline{b'})^2\equiv d \bmod{W'}.
\end{equation} 
Hence, it suffices to prove the following.

\begin{Lemma}\label{uslemma} Let $d>1$ be an integer which is not a perfect square. Then
there exist positive constants $c_1,c_2,c_3$ only depending on $d$ such that the following holds. If $Q$ is a positive integer and $x$ is a solution to the quadratic congruence
$$
x^2\equiv d\bmod{Q},
$$
then there exist $u,r\in \mathbb{Z}$ with {\rm gcd}$(u,r)=1$ such that
$$
c_1\sqrt{Q}\le |r|\le c_2\sqrt{Q}
$$
and 
$$
\left|\frac{x}{Q}-\frac{u}{r}\right| \le \frac{c_3}{r^2}.
$$
\end{Lemma}

\subsection{Relation between quadratic congruences and quadratic forms}
Our approach follows Hooley's treatment of roots of quadratic congruences in \cite[section 6]{Hoo} and makes use of the theory of binary quadratic forms.  Recall that the discriminant of the quadratic form
$$
ax^2+2hxy+by^2
$$ 
is the quantity
$$
D=h^2-ab.
$$
If the congruence 
$$
x^2\equiv  d \bmod{Q}
$$ 
has a solution $\omega$, then the quadratic form 
\begin{equation} \label{ourform}
Qx^2+2\omega xy+\frac{\omega^2-d}{Q}y^2
\end{equation}
represents $Q$ for $(x,y)=(1,0)$ and has discriminant $d$. We shall turn it into a reduced form and then express $\omega/Q$ in a way which is sufficient to prove Lemma \ref{uslemma}. 

Any form equivalent to \eqref{ourform} is obtained via a change of variables
   $$\begin{pmatrix}
   r & \beta \\
   s & \alpha
    \end{pmatrix}
    \begin{pmatrix}
    x\\ y \end{pmatrix}
    = \begin{pmatrix}
    x'\\ y'
    \end{pmatrix},
    $$  
where the matrix above is in SL$_2(\mathbb{Z})$.
In particular, $(x,y)=(1,0)$ is taken to $(x',y')=(r,s)$ via this change of variables. Therefore, any form
\begin{equation} \label{eqform}
ax^2+2hxy+by^2
\end{equation}
equivalent to \eqref{ourform} represents $Q$ in the form
\begin{equation} \label{Qrep}
ar^2+2hrs+bs^2=Q,
\end{equation}
where $r$ and $s$ are relatively prime integers such that for suitable relatively prime integers
$\alpha$ and $\beta$, one has
\begin{equation} \label{deter}
r\alpha-s\beta=1
\end{equation}
and 
\begin{equation} \label{2forms}
Qx^2+2\omega xy+\frac{\omega^2-d}{Q}y^2=
a(rx+\beta y)^2+2h(rx+\beta y)(sx+\alpha y)+
b(sx+\alpha y)^2.
\end{equation}
Comparing the coefficients of $xy$ on both sides, we calculate that
$$
\omega=ar\beta+h(r\alpha+s\beta)+bs\alpha,
$$ 
and hence, using \eqref{deter} and \eqref{Qrep},
\begin{equation} \label{firstapp}
\begin{split}
    \frac{\omega}{Q}=&\frac{ar\beta+h(r\alpha+s\beta)+bs\alpha}{ar^2+2hrs+bs^2}\\
    = & \frac{ar^2\beta+hr^2\alpha+hrs\beta+brs\alpha}{r\left(ar^2+2hrs+bs^2\right)}\\
    =& \frac{ar^2\beta+hr(s\beta+1)+hrs\beta+bs(s\beta+1)}{r\left(ar^2+2hrs+bs^2\right)}\\
    =&\frac{\beta(ar^2+2hrs+bs^2)+hr+bs}{r(ar^2+2hrs+bs^2)}\\
    =&\frac{\beta}{r}+\frac{hr+bs}{r(ar^2+2hrs+bs^2)}\\
    =&\frac{\beta}{r}+\frac{hr+bs}{rQ},
    \end{split}
\end{equation} 
where gcd$(\beta,r)=1$. 

\subsection{Evaluating the approximation}
Thus we have established the approximation \eqref{firstapp} 
for the ratio of a root $\omega$ of the quadratic congruence $x^2\equiv d \bmod{Q}$ and its modulus $Q$.
To complete the proof of Lemma \ref{uslemma}, we now assume without loss of generality that \eqref{eqform} is reduced and establish that for {\it suitable} solutions $(r,s)$ of \eqref{Qrep}, we have
$$
|r|\asymp \sqrt{Q} \quad \mbox{ and } \quad \frac{hr+bs}{rQ}\ll \frac{1}{r^2},
$$
where the implied constants are only allowed to depend on $d$.
Since there are only finitely many reduced forms of discriminant $d$, we may treat $a,h,b$ like constants, and it suffices to establish that given $a,h,b$, we have
\begin{equation} \label{firstcase}
|r|\asymp \sqrt{Q} \quad \mbox{ and }\quad  |s|\ll |r|
\end{equation}
for suitable integers $r,s$ satisfying \eqref{Qrep}, where the implied constants in \eqref{firstcase} depend only on $a,b,h$. 

We use the following information given in \cite[page 109]{Hoo}. Without loss of generality, we may take $a$ to be positive and $b$ to be negative. Let 
$$
m:=\mbox{gcd}(a,h,b)
$$
and $(T,U)$ the least solution of the Pellian equation
$$
T^2-dU^2=m.
$$
Then there exists precisely one solution $(r,s)$ of \eqref{Qrep} such that
\begin{equation*}
r>0 \quad \mbox{ and } \quad 0<s\le \frac{aU}{T-hU}\cdot r.
\end{equation*}
Moreover, on the set
$$
S:=\left\{(x,y)\in \mathbb{R}: x>0, \ 0<y\le \frac{aU}{T-hU}\cdot x\right\},
$$
the form 
$$
ax^2+2hxy+by^2
$$
takes positive values only. With this information, we are able to finish off our proof easily.

Let 
$$
\eta_1:=\min\limits_{y\in [0,aU/(T-hU)]} (a+2hy+by^2)
$$
and 
$$
\eta_2:=\max\limits_{y\in [0,aU/(T-hU)]} (a+2hy+by^2).
$$
From the above, $0<\eta_1<\eta_2$, and 
$$
\eta_1 x^2 \le ax^2+2hxy+by^2 \le \eta_2 x^2
$$ 
whenever $(x,y)\in \mathcal{S}$. Hence, $r\asymp \sqrt{Q}$, and the claim \eqref{firstcase} is established. This completes the proof of Lemma \ref{uslemma}.

\section{Dirichlet approximation in $\mathbb{Q}(\sqrt{d})$} \label{Dirichlet}
\subsection{Diophantine approximation} \label{DA}
As announced in subsection \ref{approstart}, we now work out the details of our  simultaneous Diophantine approximation of  $x_1$ and $x_2$ using Dirichlet's approximation theorem in $\mathbb{Q}(\sqrt{d})$. By Corollary \ref{diriapptheoremweak},  there is an infinite increasing sequence of natural numbers $W$ such that 
\begin{equation} \label{dirapproxi}
\begin{split}
    \left|x_1-\frac{u+v\sqrt{d}}{f+g\sqrt{d}}\right|&\leq \frac{C}{W},\\
    \left|x_2-\frac{u-v\sqrt{d}}{f-g\sqrt{d}}\right|&\leq \frac{C}{W}, 
    \end{split}
\end{equation}
where $u,v,f,g$ are rational integers, $u+v\sqrt{d}$ and $f+g\sqrt{d}$ are relatively prime in $\mathcal{O}$, 
$$
W:=\mathcal{N}(f+g\sqrt{d}),
$$
and $C$ is a positive constant only depending on $d$. 
We further write
$$
\frac{u+v\sqrt{q}}{f+g\sqrt{d}}=\frac{a}{W}+\frac{b}{W}\sqrt{d}
$$
with $a,b\in \mathbb{Z}$. 

Here we note that applying the two-dimensional version of Dirichlet's approximation theorem in $\mathbb{Q}$ would not be sufficient for our purposes. This would give us approximations of the form
$$
\left| x_i-\frac{a_{i}}{W}\right|\le \frac{1}{W^{3/2}} \quad \mbox{for } i=1,2 
$$
with $a_1,a_2,W\in \mathbb{Z}$, 
but we have no information about the Diophantine properties of the pair $(a_1,a_2)$ modulo $W$, contrary to our situation above, where we have a relation between $a$ and $b$ modulo $W$ in form of the quadratic congruence \eqref{congrurel}, which turns out to be crucial in solving our counting problem. The key point is that this allows us, after reduction, to obtain the Diophantine approximation \eqref{crucialDio} of the fraction $a'\overline{b'}/W'$ by a fraction with denominator of size $\asymp \sqrt{W'}$, which is precisely what we need to get optimal bounds in the counting problem. Therefore, it is essential to approximate inside $\mathbb{Q}(\sqrt{d})$.

\subsection{Reducing $a$, $b$ and $W$}\label{redu}
Now we want to show that $a$, $b$ and $W$ allow for a reduction in the form given in \eqref{abqcond} or \eqref{abqcond2}.
We have
\begin{equation} \label{i1}
    a+b\sqrt{d}=(u+\sqrt{d}v)(f-\sqrt{d}g),
\end{equation}
\begin{equation*} 
    W=|(f+g\sqrt{d})(f-g\sqrt{d})|
\end{equation*}
and 
\begin{equation*} 
\mbox{gcd}(u+v\sqrt{d},f+g\sqrt{d})\approx 1 \approx \mbox{gcd}(u-v\sqrt{d},f-g\sqrt{d}).
\end{equation*}
Let
$$
Z:=\mbox{gcd}(f,g) 
$$ 
and 
$$
f':=\frac{f}{Z}, \quad g':=\frac{g}{Z},\quad W':=\frac{W}{Z^2}
$$
so that gcd$(f',g')=1$.
Then by \eqref{i1}, we also have $Z|a$ and $Z|b$. Let
$$
a':=\frac{a}{Z}, \quad b':=\frac{b}{Z}.
$$
It follows that 
\begin{equation} \label{i2}  
    a'+b'\sqrt{d}=(u+v\sqrt{d})(f'-g'\sqrt{d}),
\end{equation}
\begin{equation} \label{i3}
    W'=|(f'+g'\sqrt{d})(f'-g'\sqrt{d})|
\end{equation}
and 
\begin{equation} \label{i4}
\mbox{gcd}(u+v\sqrt{d},f'+g'\sqrt{d})\approx 1 \approx \mbox{gcd}(u-v\sqrt{d},f'-g'\sqrt{d}).
\end{equation}

Assume that 
\begin{equation} \label{f'g'}
\mbox{gcd}\left(f'+g'\sqrt{d},f'-g'\sqrt{d}\right)\approx t.
\end{equation}
Then, since
$$
(f'+g'\sqrt{d})+(f'-g'\sqrt{d})=2f' \quad \mbox{ and } \quad 
(f'+g'\sqrt{d})-(f'-g'\sqrt{d})=2g'\sqrt{d},
$$
it follows that $t|2f'$ and $t|2g'\sqrt{d}$. Since gcd$(f',g')=1$ (and hence gcd$(f',g')\approx 1$ in $\mathcal{O}$), this implies
$t|2\sqrt{d}$. But then $t|\sqrt{d}$ because otherwise $2|t$ 
and hence $2|f'$ and $2|g'$ by \eqref{f'g'} which contradicts the coprimality of 
$f'$ and $g'$. Now the only possibilities are $t=\pm \sqrt{d}$ and $t=\pm 1$. In the first case, $d|f'$ and necessarily gcd$(g',d)=1$ because otherwise gcd$(f',g')\not=1$. Using the equations
\begin{align} \label{aeq}
  2a' =  (u+v\sqrt{d})(f'-g'\sqrt{d})+(u-v\sqrt{d})(f'+g'\sqrt{d})
\end{align}
and
\begin{align} \label{beq}
   2b'\sqrt{d}= (u+v\sqrt{d})(f'-g'\sqrt{d})-(u-v\sqrt{d})(f'+g'\sqrt{d})
\end{align}
(which follow from \eqref{i2}) together with \eqref{i4}, we deduce that 
\begin{align*}
 \mbox{gcd}(2a',f'\pm g'\sqrt{d}) \approx t\approx \mbox{gcd}\left(2b'\sqrt{d},f'\pm g'\sqrt{d}\right).
\end{align*}
If $t=\pm 1$, then using \eqref{i3}, it follows that 
$\mbox{gcd}(b',W')=1$. 
If $t=\pm \sqrt{d}$, then $\sqrt{d}|2a'$ and hence $d|2a'$, from which we deduce that $\mbox{gcd}(2a',W')=d$. Since we assumed that $d\equiv 3 \bmod{4}$, it follows that $\mbox{gcd}(a',W')=d$. In summary, we have either \eqref{abqcond} or \eqref{abqcond2}. 

Now we have filled all gaps in section \ref{counting}.  Further, in view of \eqref{dirapproxi}, we may take 
$$
\theta=\frac{u+v\sqrt{d}}{f+g\sqrt{d}}
$$ 
and 
\begin{equation} \label{Deltachoice}
\Delta=\frac{C}{W}
\end{equation} 
in \eqref{Diriapprox}. Hence, our estimate \eqref{TKaftercount} for the type II sum turns into
\begin{equation} \label{TKaftercount1}
\begin{split}
T_{II}\ll &
x^{6\varepsilon}\left(\left(\delta+x^{1/2}W^{-1}\right)\left(x^{1-\mu/2}Z^2
+x^{(1+\mu+\kappa)/2}Z^2+x^{1/2}W^{1/2}Z\right)+xW^{-1/2}Z+x^{3/4}\right).
\end{split}
\end{equation}
The appearance of powers of $Z$ above presents a serious problem. To overcome this, we introduced the notions of good and bad $(x_1,x_2)$ in subsection \eqref{GB}. Recalling definition \ref{goodbad}, 
if $(x_1,x_2)$ is $\eta$-good, then \eqref{TKaftercount1} yields
\begin{equation} \label{TKaftercountgood}
\begin{split}
T_{II}\ll &
x^{6\varepsilon}\left(\left(\delta+x^{1/2}W^{-1}\right)\left(x^{1-\mu/2}W^{2\eta}+x^{(1+\mu+\kappa)/2}W^{2\eta}+x^{1/2}W^{1/2+\eta}\right)+xW^{\eta-1/2}+x^{3/4}\right).
\end{split}
\end{equation}
We note that we may assume, without loss of generality, that $\eta\le 1/2$ since
$Z^2\le W$. In the next subsection we shall show that, in the sense of the Lebesgue measure, for every $\eta\in (0,1/2]$, almost all $(x_1,x_2)$ are $\eta$-good, which implies that almost all $(x_1,x_2)$ are good. 

\subsection{Almost all $(x_1,x_2)$ are good} \label{almostall}
In this subsection, we provide a measure theoretical proof of the following result. Here we point out that $\sigma(\mathbb{K})$ has Lebesgue measure 0 in 
$\mathbb{R}^2$. 

\begin{Theorem} Given $\eta>0$, almost all $(x_1,x_2)\in \mathbb{R}^2\setminus \sigma(\mathbb{K})$ are $\eta$-good.
\end{Theorem}

\begin{proof} 
Throughout this proof, we set $\mathcal{X}:= \mathbb{R}^2\setminus \sigma(\mathbb{K})$,
$
\alpha:=(x_1,x_2)
$
and 
$
\overline{n}:=\sigma_2(n)
$
if $n\in \mathcal{O}$
for convenience. If
$
q\in \mathcal{O}\setminus \{0\},
$
we set
$$
B(q):=\bigcup\limits_{p\in \mathcal{O}} \left[\frac{p}{q}-\frac{1}{F_(q)}, \frac{p}{q}+\frac{1}{F(q)}\right] \times  
\left[\frac{\overline{p}}{\overline{q}}-\frac{1}{F(\overline{q})}, \frac{\overline{p}}{\overline{q}}+\frac{1}{F(\overline{q})}\right]\cap \mathcal{X},
$$
where 
\begin{equation} \label{Fqdef}
F(q):=|q|\sqrt{\tilde{\mathcal{N}}(q)}/C
\end{equation}
with $C$ being the positive constant from Theorem \ref{diriapptheoremstrong} and 
$$
\tilde{\mathcal{N}}(q):=\sigma_1(q)^2+\sigma_2(q)^2.
$$
Further, we set
$$
Z(q):=\mbox{gcd}(f,g).
$$
Noting that $F(q)\ge \mathcal{N}(q)/C$, we observe that the set of $\eta$-good $\alpha$'s contains the set
$$
G=\bigcap\limits_{\substack{N\in \mathbb{N}\\ N\ge M}} \bigcup\limits_{\substack{q\in \mathcal{O}\setminus \{0\}\\ Z(q)\le \mathcal{N}(q)^{\eta}\\ \mathcal{N}(q)>N}} B(q),
$$
where $M$ is any positive number. Now it suffices to prove that the complement 
$G^c$ of $G$ in $\mathcal{X}$ has Lebesgue measure 0.

We further note that it is sufficient to restrict ourselves to $\alpha$'s contained in the parallelogram $P$ with sides given by the vectors $(1,1)$ and 
$(\sqrt{d},-\sqrt{d})$. This is for the following reason: If $\alpha=(x_1,x_2)\in B(q)$, then also 
$\tilde{\alpha}= (x_1+a+b\sqrt{d},x_2+a-b\sqrt{d})\in B(q)$ for any $a,b\in \mathbb{Z}$. The set
$$
\{(a+b\sqrt{d},a-b\sqrt{d}) : (a,b)\in \mathbb{Z}^2\}
$$
may be written in the form 
$$
\{a(1,1)+b(\sqrt{d},-\sqrt{d}):(a,b)\in \mathbb{Z}^2\}=\Lambda,
$$
which forms a lattice in $\mathbb{R}^2$, generated by the vectors $(1,1)$ and $(\sqrt{d},-\sqrt{d})$. Hence, it suffices to show that $P\cap G^c$ has
Lebesgue measure 0.

We may write
$$
P\cap G^c=P\setminus \Bigg(\bigcap\limits_{\substack{N\in \mathbb{N}\\ N\ge M}} \bigcup\limits_{\substack{q\in \mathcal{O}\setminus \{0\}\\ Z(q)\le \mathcal{N}(q)^{\eta}\\ \mathcal{N}(q)>N}} \tilde{B}(q)\Bigg),
$$
where 
$$
\tilde{B}(q):=P\cap B(q).
$$
Next, we write 
\begin{equation*} \begin{split}
\bigcup\limits_{\substack{q\in \mathcal{O}\setminus \{0\}\\ Z(q)\le \mathcal{N}(q)^{\eta}\\ \mathcal{N}(q)>N}} \tilde{B}(q) \supseteq &
\Bigg( \bigcup\limits_{\substack{q\in \mathcal{O}\setminus \{0\}\\ \mathcal{N}(q)>N}} \tilde{B}(q)\Bigg)\setminus 
 \Bigg(\bigcup\limits_{\substack{q\in \mathcal{O}\setminus \{0\}\\ Z(q)> \mathcal{N}(q)^{\eta}\\ \mathcal{N}(q)>N}} \tilde{B}(q)\Bigg)\\
= & \Bigg(\bigcup\limits_{\substack{q\in \mathcal{O}\setminus \{0\}\\ \mathcal{N}(q)>N}} \tilde{B}(q)\Bigg)\cap
 \Bigg(\bigcup\limits_{\substack{q\in \mathcal{O}\setminus \{0\}\\ Z(q)>\mathcal{N}(q)^{\eta}\\ \mathcal{N}(q)>N}} \tilde{B}(q)\Bigg)^c,
\end{split}
\end{equation*}
where the complement in the last line is taken inside $P\cap \mathcal{X}$. It follows that
\begin{align*} 
\bigcap\limits_{\substack{N\in \mathbb{N}\\ N\ge M}}\bigcup\limits_{\substack{q\in \mathcal{O}\setminus \{0\}\\ Z(q)\le \mathcal{N}(q)^{\eta}\\ \mathcal{N}(q)>N}} \tilde{B}(q) \supseteq
\Bigg(\bigcap\limits_{\substack{N\in \mathbb{N}\\ N\ge M}}\bigcup\limits_{\substack{q\in \mathcal{O}\setminus \{0\}\\ \mathcal{N}(q)>N}} \tilde{B}(q)\Bigg)\cap
 \Bigg(\bigcap\limits_{\substack{N\in \mathbb{N}\\ N\ge M}}\Bigg(\bigcup\limits_{\substack{q\in \mathcal{O}\setminus \{0\}\\ Z(q)> \mathcal{N}(q)^{\eta}\\ \mathcal{N}(q)>N}} \tilde{B}(q)\Bigg)^c\Bigg).
\end{align*}
Using Dirichlet's approximation theorem for $\mathbb{Q}(\sqrt{d})$, Theorem \ref{diriapptheoremstrong}, together with the above restriction-to-$P$ argument, the set
$$
\bigcap\limits_{\substack{N\in \mathbb{N}\\ N\ge M}}\bigcup\limits_{\substack{q\in \mathcal{O}\setminus \{0\}\\ \mathcal{N}(q)>N}} \tilde{B}(q)
$$
agrees with $P$ upto a set of measure 0. It follows that
\begin{align*}
\mu\Bigg(P\setminus \Bigg(\bigcap\limits_{\substack{N\in \mathbb{N}\\ N\ge M}} \bigcup\limits_{\substack{q\in \mathcal{O}\setminus \{0\}\\ Z(q)\le \mathcal{N}(q)^{\eta}\\ \mathcal{N}(q)>N}} \tilde{B}(q)\Bigg)\Bigg)\le \mu\Bigg(\bigcup\limits_{\substack{N\in \mathbb{N}\\ N\ge M}}\bigcup\limits_{\substack{q\in \mathcal{O}\setminus \{0\}\\ Z(q)> \mathcal{N}(q)^{\eta}\\ \mathcal{N}(q)>N}} \tilde{B}(q)\Bigg) = \mu\Bigg(\bigcup\limits_{\substack{q\in \mathcal{O}\setminus \{0\}\\ Z(q)> \mathcal{N}(q)^{\eta}\\ \mathcal{N}(q)>M}}  \tilde{B}(q)\Bigg),
\end{align*}
where $\mu$ is the Lebesgue measure.
Further,
$$
\bigcup\limits_{\substack{q\in \mathcal{O}\setminus \{0\}\\ Z(q)> \mathcal{N}(q)^{\eta}\\ \mathcal{N}(q)>M}} \tilde{B}(q)=\bigcup\limits_{Z>M^{\eta}} 
\bigcup\limits_{\substack{q\in \mathcal{O}\setminus \{0\}\\ M<\mathcal{N}(q)<Z^{1/\eta}\\ Z(q)=Z}} \tilde{B}(q)
$$
and, setting $q'=q/Z$ and $q'=f'+g'\sqrt{d}$,
\begin{equation*}
\begin{split}
& \bigcup\limits_{\substack{q\in \mathcal{O}\setminus \{0\}\\ M<\mathcal{N}(q)<Z^{1/\eta}\\ Z(q)=Z}} \tilde{B}(q)\\
= &  \bigcup\limits_{\substack{q'\in \mathcal{O}\setminus \{0\}\\ M/Z^2<\mathcal{N}(q')<Z^{1/\eta-2}\\ (f',g')=1}}  \bigcup\limits_{p\in \mathcal{O}}\ P\cap\Bigg(\left[\frac{p}{Zq'}-\frac{1}{Z^2F(q')}, \frac{p}{Zq'}+\frac{1}{Z^2F(q')}\right] \times  \\ & \quad
\quad
\left[\frac{\overline{p}}{Z\overline{q'}}-\frac{1}{Z^2F(\overline{q'})}, \frac{\overline{p}}{Z\overline{q'}}+\frac{1}{Z^2F(\overline{q'})}\right]\Bigg)\\
\subseteq & \frac{1}{Z}  \bigcup\limits_{\substack{q\in \mathcal{O}\setminus \{0\}\\ \mathcal{N}(q)<Z^{1/\eta-2}}}\bigcup\limits_{p\in \mathcal{O}}\ (ZP)\cap \Bigg( \left[\frac{p}{q}-\frac{1}{ZF(q)}, \frac{p}{q}+\frac{1}{ZF(q)}\right] \times  
\left[\frac{\overline{p}}{\overline{q}}-\frac{1}{ZF(\overline{q})}, \frac{\overline{p}}{\overline{q}}+\frac{1}{ZF(\overline{q})}\right]\Bigg),
\end{split}
\end{equation*}
where in the last line, we have replaced $q'$ by $q$. 
Given $q$, the number of $p$'s such that $ZP$ has non-empty intersection with the rectangle
$$
R_Z(p,q)=\left[\frac{p}{q}-\frac{1}{ZF(q)}, \frac{p}{q}+\frac{1}{ZF(q)}\right] \times  
\left[\frac{\overline{p}}{\overline{q}}-\frac{1}{ZF(\overline{q})}, \frac{\overline{p}}{\overline{q}}+\frac{1}{ZF(\overline{q})}\right]
$$
is bounded by $O(Z^2 \mathcal{N}(q))$. To see this, note that
$$
F(q), F(\overline{q})\ge \mathcal{N}(q)
$$
for all $q$, and for fixed $q$, the points 
$$
\left(\frac{p}{q},\frac{\overline{p}}{\overline{q}}\right), \quad p\in \mathcal{O}
$$
form a lattice $\Lambda(q)$ in $\mathbb{R}^2$ which contains $\Lambda=\Lambda(1)$ as a sublattice of index $\mathcal{N}(q)$ in $\Lambda(q)$. Moreover,
$$
\mu(R_Z(p,q))=\frac{1}{Z^2F(q)F(\overline{q})}=\frac{C^2}{Z^2\mathcal{N}(q)\tilde{\mathcal{N}}(q)}.
$$

Combining everything and using countable subadditivity, we arrive at 
$$
\mu(P\cap G^c)\ll \sum\limits_{Z>M^{\eta}} \sum\limits_{\substack{q
\in \mathcal{O}\setminus \{0\}\\ \mathcal{N}(q)<Z^{1/\eta-2}}} \frac{1}{Z^2\tilde{\mathcal{N}}(q)}.
$$
By Lemma \ref{supplement}, for any given non-zero (principal) ideal $\mathfrak{a}$, we have
$$
\sum\limits_{\substack{q\in \mathcal{O}\\ (q)=\mathfrak{a}}} \frac{1}{\tilde{\mathcal{N}}(q)} \ll \frac{1}{\mathcal{N}(\mathfrak{a})}.
$$
Hence, 
$$
\mu(P\cap G^c)\ll \sum\limits_{Z>M^{\eta}} \frac{1}{Z^2} \sum\limits_{\substack{\mathfrak{a}\in I\setminus 0\\ \mathcal{N}(\mathfrak{a})<Z^{1/\eta-2}}} \frac{1}{\mathcal{N}(\mathfrak{a})}\ll_{\eta} \sum\limits_{Z>M^{\eta}} \frac{\log Z}{Z^2} \ll M^{-\eta/2}.
$$
Since this holds for all positive $M$, it follows that
$$
m(P\cap G^c)=0.
$$
This completes the proof.
\end{proof}

We deduce the following.

\begin{Corollary} Almost all $(x_1,x_2)\in \mathbb{R}^2$ are good.
\end{Corollary}
\begin{proof}
Let $\mathcal{G}_{\eta}$ be the set of $\eta$-good $(x_1,x_2)$ and $\mathcal{G}$ be the set of good $(x_1,x_2)$. Then 
$$
\mathcal{G}=\bigcap\limits_{n=1}^{\infty} \mathcal{G}_{1/n}.
$$
Hence,
$$
\mathcal{G}^c=\bigcup\limits_{n=1}^{\infty} \mathcal{G}_{1/n}^c
$$
and therefore
$$
\mu(\mathcal{G}^c)\le \sum\limits_{n=1}^{\infty} \mu(\mathcal{G}_{1/n}^c)=0,
$$
which completes the proof. 
\end{proof}

\subsection{Constructing good $(x_1,x_2)$} \label{consgood}
After we have seen that almost all $(x_1,x_2)$ are good, we supplement an explicit construction of good $(x_1,x_2)$ using continued fractions.
As in the previous subsection, if $a\in \mathcal{O}$, let $\overline{a}\in \mathcal{O}$ be the conjugate of $a$, i.e., $\overline{a}:=\sigma_2(a)$. We pick a sequence $(a_k)_{k\in \mathbb{N}\cup \{0\}}$ in $\mathcal{O}$ satisfying the condition
$$
\inf\limits_{k\in \mathbb{N}} a_k>0 \quad \mbox{ and }\quad  \inf\limits_{k\in \mathbb{N}} \overline{a_k}>0
$$
which we assume to hold throughout this subsection. In particular, $a_k,\overline{a_k}>0$ for all $k\ge 0$. 
Under this condition, the continued fractions
\begin{equation} \label{x1x2cont}
x_1=a_0+\cfrac{1}{a_1+\cfrac{1}{a_2+\cdots}} \quad \mbox{ and } \quad x_2=
\overline{a_0}+\cfrac{1}{\overline{a_1}+\cfrac{1}{\overline{a_2}+\cdots}}
\end{equation}
both converge. Our idea is to choose $(a_k)_{k\in \mathbb{N}\cup \{0\}}$ in such a way that the convergents of these continued fractions produce
sequences
$$
\left(\frac{u_k+v_k\sqrt{d}}{f_k+g_k\sqrt{d}},\frac{u_k-v_k\sqrt{d}}{f_k-g_k\sqrt{d}}\right)
$$ 
of good simultaneous approximations of $(x_1,x_2)$ with gcd$(f_k,g_k)=1$, thus establishing that $(x_1,x_2)$ is good.

Similarly as in the context of continued fractions for integers, we write  
$$
a_0+\cfrac{1}{a_1+\cfrac{1}{a_2+\cdots}}=:[a_0:a_1,a_2,...]
$$
and its convergents as
$$
a_0+\cfrac{1}{a_1+\cfrac{1}{a_2+\cfrac{}{\ddots \cfrac{1}{a_k}}}}=:[a_0:a_1,a_2,...,a_k]=\frac{p_k}{q_k}=\frac{u_k+v_k\sqrt{d}}{f_k+g_k\sqrt{d}} \quad
\mbox{ for } k\ge 0,
$$
where $p_k,q_k\in \mathcal{O}$ are coprime and $u_k,v_k,f_k,g_k\in \mathbb{Z}$. Here, $p_k$ and $q_k$ are unique only up to units. Clearly, we have
$$
\frac{\overline{p_k}}{\overline{q_k}}=[\overline{a_0}:\overline{a_1},\overline{a_2},...,\overline{a_k}].
$$ 
The results in the following lemma are standard in the context of continued fractions for rational integers and can be proved in our setting of algebraic integers
in $\mathcal{O}$ in an analog way. For proofs in the classical setup, we refer the reader to \cite{Khin}. 

\begin{Lemma} \label{contfract}
{\rm (i)} For $k\geq 2$, assume that
$$
\frac{p_l}{q_l}=[a_0:a_1,a_2,...,a_l] \quad \mbox{ if } l=k-2,k-1.
$$
Then
$$
\frac{p_k}{q_k}=[a_0:a_1,a_2,...,a_k]
$$
for
\begin{equation} \label{recur}
\begin{split}
p_k= & a_kp_{k-1}+p_{k-2},\\
q_k= & a_kq_{k-1}+q_{k-2}.
\end{split}
\end{equation}
{\rm (ii)} If the sequences $(p_k)_{k\in \mathbb{N}}$ and $(q_k)_{k\in \mathbb{N}}$ satisfy the recursive relation \eqref{recur} for all $k\ge 2$ and, without loss of generality, 
$$
p_0=a_0, \quad q_0=1, \quad p_1=a_0a_1+1, \quad q_1=a_1, 
$$
then we have $q_k>0$ for all $k\ge 0$, 
\begin{equation*}
q_kp_{k-1}-p_kq_{k-1}=(-1)^k
\end{equation*}
for all $k\ge 1$ and, consequently,
$$
\Bigg|x_1-\frac{p_{k-1}}{q_{k-1}}\Bigg|\le \frac{1}{q_kq_{k-1}} \quad \mbox{ and } \quad \Bigg|x_2-\frac{\overline{p_{k-1}}}{\overline{q_{k-1}}}\Bigg|\le \frac{1}{\overline{q_kq_{k-1}}}
$$
for all $k\ge 1$.
\end{Lemma}

Now we begin with our recursive construction. In each step, we have to ensure that gcd$(f_k,g_k)=1$, and we also need a strong enough growth of $q_k$ and $\overline{q_k}$ to get sufficiently strong approximations. For the start, we take 
$$
a_0:=1 \quad \mbox{ and } \quad a_1:=f_1+g_1\sqrt{d} \mbox{ with } f_1>dg_1>0 \mbox{ and gcd}(f_1,g_1)=1.
$$
Then, as in Lemma \ref{contfract}(i), we may take 
$$
p_0=1, \ q_0=f_0+g_0\sqrt{d}=1, \ p_1=f_1+1+g_1\sqrt{d}, \ q_1=f_1+g_1\sqrt{d}.
$$
We check that for $k=2$, the following conditions hold. 
\begin{equation} \label{contconds}
\begin{split}
\mbox{gcd}(f_{k-2},g_{k-2})=&1,\\
\mbox{gcd}(f_{k-1},g_{k-1})=&1,\\
\mbox{gcd}(f_{k-1},f_{k-2})=&1,\\ 
f_{k-2}> dg_{k-2} > & 0,\\
f_{k-1}> dg_{k-1} > &0,\\
f_{k-1}> & f_{k-2}^2.
\end{split}
\end{equation}
Assume this holds for $k\ge 2$. We shall construct $a_k,p_k,q_k$ in such a way that the conditions in \eqref{contconds} hold with $k+1$ in place of $k$, i.e.,
 \begin{equation} \label{contconds+1}
\begin{split}
\mbox{gcd}(f_{k-1},g_{k-1})=&1,\\
\mbox{gcd}(f_{k},g_{k})=&1,\\
\mbox{gcd}(f_{k},f_{k-1})=&1,\\ 
f_{k-1}>dg_{k-1}>& 0,\\
f_{k}>dg_{k}> & 0,\\
f_{k}> & f_{k-1}^2.
\end{split}
\end{equation}

In the following, we write
$$
a_k=s_k+t_k\sqrt{d}, \quad \mbox{ where } s_k,t_k\in \mathbb{Z}.
$$
According to Lemma \ref{contfract}(i), we may then take
\begin{align*}
q_k=&a_kq_{k-1}+q_{k-2}\\
     =&(s_k + t_k \sqrt{d})(f_{k-1}+g_{k-1}\sqrt{d})+(f_{k-2}+g_{k-2}\sqrt{d})\\
     =&(s_k f_{k-1}+t_k g_{k-1} d + f_{k-2}) +(t_kf_{k-1}+s_kg_{k-1}+g_{k-2})\sqrt{d}
\end{align*}
and hence 
\begin{equation*}
\begin{split}
f_k=&s_k f_{k-1}+t_k g_{k-1} d + f_{k-2},\\
g_k=&t_k f_{k-1}+s_kg_{k-1}+g_{k-2}.
\end{split}
\end{equation*}
If $s_k,t_k>0$, which we want to assume from now on, then it follows from $f_{k-2},g_{k-1}>0$ that 
$f_k>f_{k-1}$. So if $f_k$ is prime, then gcd$(f_k,f_{k-1})=1$. If, in addition, $f_k>dg_k$, which is the case 
if
\begin{equation} \label{>condi}
s_k(f_{k-1}-dg_{k-1})+(f_{k-2}-dg_{k-2})>d (f_{k-1}-g_{k-1})t_k,
\end{equation}
then gcd$(f_k,g_k)=1$. So to establish \eqref{contconds+1}, all we need is to
find $s_k$ and $t_k$ such that \eqref{>condi} holds and $f_k$ is a prime satisfying $f_k>f_{k-1}^2$.

Since gcd$(f_{k-1},f_{k-2})=1$, Dirichlet's theorem on the infinitude of primes in arithmetic progressions ensures the existence of $\tilde{s}_k>0$ such that
$\tilde{s}_k f_{k-1} + f_{k-2}$ is prime and $\tilde{s}_k f_{k-1} + f_{k-2}>dg_{k-1}$. So gcd$(\tilde{s}_k f_{k-1} + f_{k-2},dg_{k-1})=1$, and by the same theorem we can choose $t_k$ 
in such a way that $\tilde{s}_k f_{k-1}+t_k g_{k-1} d + f_{k-2}$ is prime. The pair $(s_k,t_k)=(\tilde{s}_k,t_k)$ may not satisfy condition \eqref{>condi}. To enforce this situation, we observe that since $\tilde{s}_k f_{k-1}+t_k g_{k-1} d + f_{k-2}$ is prime, we have gcd$(\tilde{s}_kf_{k-1},t_k g_{k-1} d + f_{k-2})=1$, and so again by Dirichlet's above-mentioned theorem, we can find $\lambda>0$ such that $\lambda \tilde{s}_k f_{k-1}+t_k g_{k-1} d + f_{k-2}$ is a prime greater than $f_{k-1}^2$ and  
$$
\lambda \tilde{s}_k(f_{k-1}-dg_{k-1})+(f_{k-2}-dg_{k-2})>d (f_{k-1}-g_{k-1})t_k.
$$
Here we recall that $f_{k-1}>dg_{k-1}$ and $f_{k-2}>dg_{k-2}$ by \eqref{contconds}.
Now choosing $s_k:=\lambda \tilde{s}_k$, the condition \eqref{>condi} is satisfied and $f_k$ is a prime greater than $f_{k-1}^2$. Hence, the conditions in \eqref{contconds+1} are established. 

We aim to show that under this construction, the pair $(x_1,x_2)$ given by \eqref{x1x2cont} is good. Since the construction gives $(f_k,g_k)=1$
for all $k$, it suffices to show that it satisfies \eqref{dirapproxi} for all $(p,q)=(p_k,q_k)$ with $k$ large enough, and that $\mathcal{N}(q_n)\rightarrow \infty$
as $n\rightarrow \infty$.  To this end, we observe that the condition $f_k\ge \max\{dg_k,f_{k-1}^2\}$ for all $k\ge 1$ implies
$$
q_k^2\asymp \mathcal{N}(q_k)\asymp \overline{q_k}^2
$$
and 
$$
q_k\gg q_{k-1}^2 \quad \mbox{ and } \quad \overline{q_{k}}\gg \overline{q_{k-1}}^2
$$
for all $k\ge 1$. 
Now the above claim follows from Lemma \ref{contfract}(ii), which establishes that $(x_1,x_2)$ is good. 

\section{Bounding the type I sum} \label{typeI}
\subsection{Initial transformations}
We are left with bounding the type I sum $T_{I}$ arising from our application of Harman's sieve in section \ref{Harmansieve}. Proceeding precisely as in the initial treatment of the type II sum $T_{II}$ in subsections \ref{It} and \ref{Cs}, but leaving the $n$-summation uncut, we arrive at
\begin{equation} \label{T1arrive}
\begin{split}
T_{I}\ll & (\log x) \delta^2\sup\limits_{1/2\le K\le M} \sum\limits_{1\le j\le 2\mathcal{C}} |\Sigma_{K,j}'|+O\left(x^{-100}\right),
\end{split}
\end{equation}
where
$$
\Sigma_{K,j}':=\sum\limits_{\substack{m\in \mathcal{R}\\K\le \mathcal{N}(m)\le 2K}} \sum\limits_{n\in \mathcal{O}\setminus \{0\}} a_{m} \mathop{\sum}\limits_{\substack{p\in \mathcal{O}\setminus\{0\}\\|\sigma_{1,2}(p)|\le x^{\varepsilon}\delta^{-1}}} E_j(p,mn)\mathcal{E}(pmn)
$$
with 
\begin{equation*}
\mathcal{E}(l):=e\left(\frac{\sigma_2(l)x_2-\sigma_1(l)x_1}{2\sqrt{d}}\right)
\end{equation*}
and
$$
E_j(p,k):=  
\exp\left(-\pi\cdot \frac{\sigma_1(pk)^2+\sigma_2(pk)^2}{4d}\cdot \frac{\delta^2}{N}\right)\cdot \exp\left(-\pi\cdot \frac{j_1\sigma_1(k)^2+j_2\sigma_2(k)^2}{N}\right).
$$ 
In the following, we bound $\Sigma_{K,j}'$. Unlike in subsection \ref{removal}, we here don't remove the weight functions but rather make use of them in directly applying the Poisson summation formula
to the {\it smooth} sum over $n$, which will be carried out in the next subsection. We still pull in the sum over $n$ and use $|a_m|\le 1$ to obtain 
\begin{equation} \label{SigmaK'}
\Sigma_{K,j}'\le \sum\limits_{\substack{m\in \mathcal{R}\\K\le \mathcal{N}(m)\le 2K}}   \sum\limits_{\substack{p\in \mathcal{O}\setminus\{0\}\\|\sigma_{1,2}(p)|\le x^{\varepsilon}\delta^{-1}}} \left|\sum\limits_{n\in \mathcal{O}\setminus \{0\}} E_j(p,mn)\mathcal{E}(pmn)\right|.
\end{equation}

\subsection{Applying Poisson summation}
Now we use Poisson summation formula to transform the sum over 
$n$ above. For $(u,v)\in \mathbb{R}^2$ set 
\begin{equation*}
\begin{split}
& f(u,v):= e\left(-\frac{\sigma_1(pm)x_1}{2\sqrt{d}}\cdot (u+v\sqrt{d})+\frac{\sigma_2(pm)x_2}{2\sqrt{d}}\cdot (u-v\sqrt{d})\right)\times\\ & \exp\left(-\pi \cdot \left(\left(\frac{\sigma_1(pm)^2\delta^2}{4dN}+\frac{j_1\sigma_1(m)^2}{N}\right)(u+v\sqrt{d})^2+
\left(\frac{\sigma_2(pm)^2\delta^2}{4dN}+\frac{j_2\sigma_2(m)^2}{N}\right)(u-v\sqrt{d})^2\right)\right) 
\end{split}
\end{equation*}
so that
$$
E_j(p,mn)\mathcal{E}(pmn)=f(u,v)
$$
if $n=u+v\sqrt{d}$ with $(u,v)\in \mathbb{Z}^2$. Clearly, $f\in L^1(\mathbb{R}^2)$. Making the linear change of variables
\begin{equation*}
\begin{split}
u+v\sqrt{d}= & \sigma,\\
u-v\sqrt{d} = & \tau,
\end{split}
\end{equation*}
we calculate the Fourier transform of $f$ to be 
\begin{equation*}
\begin{split}
\hat{f}(\alpha,\beta)= & 2\sqrt{d}\left(\sigma_1(pm)^2\delta^2/N+4dj_1\sigma_1(m)^2/N\right)^{-1/2}\left(\sigma_2(pm)^2\delta^2/N+4dj_2\sigma_2(m)^2/N\right)^{-1/2}\times\\ &
 \exp\left(-\pi \cdot \left(\frac{\left(\beta+\alpha\sqrt{d}+\sigma_1(pm)x_1\right)^2}{\sigma_1(pm)^2\delta^2/N+4dj_1\sigma_1(m)^2/N}+\frac{\left(-\beta+\alpha\sqrt{d}-
 \sigma_2(pm)x_2\right)^2}{\sigma_2(pm)^2\delta^2/N+4dj_2\sigma_2(m)^2/N}\right)\right)\\
 \ll & \frac{N}{K} \cdot \exp\left(-\pi \cdot \frac{N}{Kx^{3\varepsilon}}\cdot \left(\left(\beta+\alpha\sqrt{d}+\sigma_1(pm)x_1\right)^2+\left(-\beta+\alpha\sqrt{d}-\sigma_2(pm)x_2\right)^2\right)\right)
  \end{split}
\end{equation*}
if $m$ and $p$ satisfy the relevant summation conditions in \eqref{SigmaK'}.  Obviously, the above function $\hat{f}$ is also in $L^1(\mathbb{R}^2)$ and hence, we may use Lemma \ref{Poisson} to deduce that
\begin{equation} \label{RHS}
\begin{split}
& \sum\limits_{n\in \mathcal{O}} E_j(p,mn)\mathcal{E}(pmn)= \sum\limits_{(\alpha,\beta)\in \mathbb{Z}^2} \hat{f}(\alpha,\beta)\\
\ll & \frac{x}{K} \cdot \sum\limits_{(\alpha,\beta)\in \mathbb{Z}^2}\exp\left(-\pi \cdot \frac{N}{Kx^{3\varepsilon}}\cdot \left(\left(\beta+\alpha\sqrt{d}+\sigma_1(pm)x_1\right)^2+\left(-\beta+\alpha\sqrt{d}-\sigma_2(pm)x_2\right)^2\right)\right).
\end{split}
\end{equation}
The right-hand side is negligible unless
$$
\left|\beta+\alpha\sqrt{d}+\sigma_1(pm)x_1\right|\le x^{2\varepsilon-1/2}K^{1/2}
$$
and 
$$
\left|-\beta+\alpha\sqrt{d}-\sigma_2(pm)x_2 \right|\le x^{2\varepsilon-1/2}K^{1/2}.
$$
The above inequalities imply
$$
\left|\beta+\frac{\sigma_2(pm)x_2+\sigma_1(pm)x_1}{2}\right|\le x^{2\varepsilon-1/2}K^{1/2}
$$
and
$$
\left|\alpha- \frac{\sigma_2(pm)x_2-\sigma_1(pm)x_1}{2\sqrt{d}}\right|\le x^{2\varepsilon-1/2}K^{1/2}
$$
upon subtracting and adding and dividing by $2$ and $2\sqrt{d}$, respectively.
Hence, the right-hand side of \eqref{RHS} is negligble unless
$$
\left|\left|\frac{\sigma_2(pm)x_2+\sigma_1(pm)x_1}{2}\right|\right| \le x^{2\varepsilon-1/2}K^{1/2} \quad \mbox{ and } \quad 
\left|\left|\frac{\sigma_2(pm)x_2-\sigma_1(pm)x_1}{2\sqrt{d}}\right|\right| \le x^{2\varepsilon-1/2}K^{1/2},
$$
in which case it is bounded by $O(x/K)$. Hence, we deduce from \eqref{SigmaK'} that
\begin{equation*}
\begin{split}
& \Sigma_{K,j}'\ll x^{2\varepsilon}K\delta^{-2}+xK^{-1}\times\\ & 
\sum\limits_{\substack{m\in \mathcal{R}\\K\le \mathcal{N}(m)\le 2K}}
\sum\limits_{\substack{p\in \mathcal{O}\setminus\{0\}\\|\sigma_{1,2}(p)|\le x^{\varepsilon}\delta^{-1}}}  
\chi_{J'}\left(\left|\left|\frac{\sigma_2(pm)x_2+\sigma_1(pm)x_1}{2}\right|\right|\right)\chi_{J'}\left(\left|\left|\frac{\sigma_2(pm)x_2-\sigma_1(pm)x_1}{2\sqrt{d}}\right|\right|\right)
\end{split}
\end{equation*}
where the first summand on the right-hand side comes from the contribution of $n=0$, and $J'$ is the interval
\begin{equation} \label{J'def}
J':=\left[-x^{2\varepsilon-1/2}K^{1/2},x^{2\varepsilon-1/2}K^{1/2}\right].
\end{equation}
Writing $k=pm$ similarly as in subsection \ref{CauSch}, we deduce that
\begin{equation} \label{precount2}
\begin{split}
& \Sigma_{K,j}'
\ll x^{2\varepsilon}K\delta^{-2}+x^{1+\varepsilon}K^{-1}\times\\ & 
\sum\limits_{\substack{k\in \mathcal{O}\setminus\{0\}\\|\sigma_{1,2}(k)|\le cx^{\varepsilon}\delta^{-1}K^{1/2}}}   
\chi_{J'}\left(\left|\left|\frac{\sigma_2(k)x_2+\sigma_1(k)x_1}{2}\right|\right|\right)\chi_{J'}\left(\left|\left|\frac{\sigma_2(k)x_2-\sigma_1(k)x_1}{2\sqrt{d}}\right|\right|\right)
\end{split}
\end{equation}
for some constant $c>0$.

\subsection{Estimating the type I sum}
The bound \eqref{precount2} has the same shape as \eqref{precount}, and we can therefore directly apply our method in section \ref{counting} to bound this sum. In place of \eqref{simplewrite}, we now obtain
\begin{equation*} 
\Sigma_{K,j}'\ll  x^{2\varepsilon}K\delta^{-2}+x^{1+\varepsilon}K^{-1}\cdot  
\Theta(2U';a,b;W)\cdot \Theta(V';a,b;W),
\end{equation*}
where 
$$
U':=cx^{\varepsilon}\delta^{-1}K^{1/2} \quad \mbox{ and } \quad V'=W(x^{2\varepsilon-1/2}K^{1/2}+cx^{\varepsilon}\delta^{-1}K^{1/2}\Delta).
$$
The choices of $U'$ and $V'$ above are due to the fact that $k$ is of the form
\begin{equation*}
k=\lambda+\mu\sqrt{d}
\end{equation*}
with 
\begin{equation} \label{lambdamu}
|\lambda|,|\mu| \le U'.
\end{equation}

We recall \eqref{Thetaestfinal} which states that 
\begin{equation*}
\Theta(X;a,b;W)\ll  \left(\frac{X^2Z^2}{W}+1\right)\log^2 W,
\end{equation*}
and $\Delta=C/W$ (see \eqref{Deltachoice}). Hence, we have
\begin{equation} \label{largeK} 
\begin{split}
\Sigma_{K,j}'\ll  & x^{2\varepsilon}K\delta^{-2}+x^{1+8\varepsilon}K^{-1} 
\left(\delta^{-2}KW^{-1}Z^2+1\right)\left(x^{-1}KWZ^2+\delta^{-2}KW^{-1}Z^2+1\right)\\
\ll & x^{2\varepsilon}K\delta^{-2}+x^{8\varepsilon}
\left(K\delta^{-2}Z^4+K\delta^{-4}xW^{-2}Z^4+\delta^{-2}xW^{-1}Z^2+WZ^2+xK^{-1}\right)
\end{split}
\end{equation}
under the condition \eqref{qcond} which states that $W\ll x^{100}$.
 
The bound for $\Sigma_{K,j}'$ which we need to beat is $\ll x$. Therefore, if $K$ is small, \eqref{largeK} will not suffice. However, we can exclude small $K$, as we shall see in the next subsection.

\subsection{Excluding small $K$}  
Repeating our method in section \ref{counting}, 
we are led to counting solutions $(\alpha,\beta,\lambda,\mu)$ of the system of congruences
\begin{equation} \label{8.1''}
\begin{cases}
    a\lambda+bd\mu & \equiv \alpha \bmod{W},\\
    b\lambda+a\mu & \equiv \beta \bmod{W},
\end{cases}
\end{equation}
which is \eqref{8.1} in subsection \ref{countsol},
subject to the conditions 
$$
|\lambda|,|\mu|\le U' \quad \mbox{ and } \quad |\alpha|,|\beta|\le V'.
$$
Using the reduction in subsection \ref{additive}, \eqref{8.1''} implies
\begin{equation} \label{redsystem}
\begin{cases}
    a'\lambda+b'd\mu & \equiv \alpha \bmod{W'},\\
    b'\lambda+a'\mu & \equiv \beta \bmod{W'},
\end{cases}
\end{equation}
where $a'=a/Z$, $b'=b/Z$, $W'=W/Z^2$ and gcd$(b',W')=1$. Moreover, we deduced the congruence \eqref{8.2}, which states that
\begin{equation*} 
    b \alpha  \equiv a \beta  \bmod{W}. 
\end{equation*}
By the same reduction, the above congruence implies
\begin{equation} \label{afterred}
    \alpha  \equiv a'\overline{b'} \beta  \bmod{W'},
\end{equation}
where $b'$ is a multiplicative inverse modulo $W'$.
We recall that, according to \eqref{crucialDio},
\begin{equation} \label{crucialDio2}
\left| \frac{a'\overline{b'}}{W'} -\frac{u}{r} \right| \ll r^{-2} \mbox{ with gcd}(u,r)=1
\end{equation}
for some integer $r$ with  $r\asymp\sqrt{W'}$. Therefore, there is a constant $c_1>0$ such that \eqref{afterred} has no solutions $(\alpha,\beta)$
with $|\alpha|,|\beta|\le c_1\sqrt{W'}$ except for the trivial one, which is
$(\alpha,\beta)=(0,0)$. Hence, if $V'\le c_1\sqrt{W'}$ , then \eqref{redsystem}
becomes a homogeneous system
\begin{equation*} 
\begin{cases}
    a'\lambda+b'd\mu & \equiv 0 \bmod{W'},\\
    b'\lambda+a'\mu & \equiv 0 \bmod{W'}.
\end{cases}
\end{equation*}
From the second congruence, it follows that
$$
a'\overline{b'}\mu \equiv - \lambda \bmod{W'}.
$$
Again using \eqref{crucialDio2} and $|\mu|,|\lambda| \le U'$ by \eqref{lambdamu}, this congruence has only the solution $(\lambda,\mu)=(0,0)$ in the above range if $U'\le c_1
\sqrt{W'}$. But since $k\not=0$, this solution is excluded. In summary, the system \eqref{8.1'} has no solutions if 
$$
U'=cx^{\varepsilon}\delta^{-1}K^{1/2}\le c_1\sqrt{W'}=c_1\sqrt{W}/Z 
$$
and 
$$
V'=W(x^{2\varepsilon-1/2}K^{1/2}+Ccx^{\varepsilon}\delta^{-1}K^{1/2}W^{-1})\le c_1\sqrt{W'}=c_1\sqrt{W}/Z.
$$
Therefore, the sum over $k$ in \eqref{precount2} is empty and thus
\begin{equation} \label{precount4}
\Sigma_{K,j}'\ll x^{2\varepsilon}\delta^{-2}K
\quad \mbox{ if }
K\le c_2\min\left\{x^{-2\varepsilon}\delta^2WZ^{-2},x^{1-4\varepsilon}W^{-1}Z^{-2}\right\}
\end{equation}
for some constant $c_2>0$.
Combining \eqref{largeK} and \eqref{precount4}, we get the estimate
\begin{equation*} 
\begin{split}
\Sigma_{K,j}'
\ll x^{12\varepsilon}
\left(K\delta^{-2}Z^4+K\delta^{-4}xW^{-2}Z^4+\delta^{-2}xW^{-1}Z^2+WZ^2\right)
\end{split}
\end{equation*}
Plugging this into \eqref{T1arrive}, we obtain
\begin{equation*}
T_{I}\ll  x^{13\varepsilon}\left(MZ^4+M\delta^{-2}xW^{-2}Z^4+xW^{-1}Z^2+\delta^2WZ^2\right).
\end{equation*}
If $(x_1,x_2)$ is $\eta$-good, then this gives the final bound
\begin{equation} \label{T1finalgood}
T_{I}\ll  x^{13\varepsilon}\left(MW^{4\eta}+M\delta^{-2}xW^{4\eta-2}+xW^{2\eta-1}+\delta^2W^{2\eta+1}\right).
\end{equation}

\section{Conclusion}
Now we apply Theorem \eqref{T4.1}. Recalling the definitions of $\omega(\mathfrak{q})$ and $\tilde{\omega}(\mathfrak{q})$ in \eqref{2.1} and \eqref{2.2} and the bound \eqref{omegabounds}, we have
\begin{align*} 
  \lim_{R\to \infty} \sum_{\substack{\mathfrak{a}\in \mathcal{I} \\ \mathcal{N}(\mathfrak{a})<R}} d_4(\mathfrak{a})w(\mathfrak{a})\ll N^{1+\varepsilon}\ll x
\end{align*}
if $w=\omega$ or $w=\tilde{\omega}$. Now combining \eqref{thediff}, \eqref{5.1}, \eqref{5.2},\eqref{TKaftercountgood} and \eqref{T1finalgood} and choosing
$$
\mu=\frac{1}{4}, \quad \kappa=\frac{1}{2}, \quad M=2x^{1/4}, 
$$ 
we obtain
\begin{equation} \label{diffes} 
\begin{split}
& x^{-13\varepsilon}(\mathcal{T}(N)-\tilde{\mathcal{T}}(N))\\
\ll & x^{1/4}W^{4\eta}+\delta^{-2}x^{5/4}W^{4\eta-2}+xW^{2\eta-1}+\delta^2W^{2\eta+1}
+\\ &  \left(\delta+x^{1/2}W^{-1}\right)\left(x^{7/8}W^{2\eta}+x^{1/2}W^{1/2+\eta}\right)+xW^{\eta-1/2}+x^{3/4}.
\end{split}
\end{equation}
Now we optimize the parameters. First, we choose $x$ depending on $W$ in such a way $\delta=x^{1/2}W^{-1}$, i.e.  
$$
x:=(\delta W)^2
$$
and hence
$$
W\asymp x^{1/2}\delta^{-1}.
$$
Then \eqref{diffes} turns into
\begin{equation*} 
\begin{split}
& x^{-13\varepsilon}(\mathcal{T}(N)-\tilde{\mathcal{T}}(N))\\
\ll & x^{1/4+2\eta}\delta^{-4\eta}+
x^{7/8+\eta}\delta^{1-2\eta}+x^{3/4+\eta/2}\delta^{1/2-\eta}
+x^{3/4}.
\end{split}
\end{equation*}
Recalling $N:=\left\lceil x^{1-\varepsilon}\right\rceil$, the estimate 
$$
\mathcal{T}(N)-\tilde{\mathcal{T}}(N)\ll \delta^2N^{1-\varepsilon}
$$
in Theorem \ref{difftheorem} holds if $\varepsilon\le 1/14$ and 
$$ 
\delta\ge N^{-\nu+15\varepsilon}
$$
with 
$$
\nu:=\min\left\{\frac{3/4-2\eta}{2+4\eta},\frac{1/8-\eta}{1+2\eta},\frac{1/4-\eta/2}{3/2+\eta},\frac{1}{8}\right\}=
\frac{1/8-\eta}{1+2\eta}
$$
This establishes Theorem \ref{difftheorem}. 
 
\section{Unsmoothing} \label{unsmoothing}
Finally, we shall use Theorem \ref{difftheorem} and Corollary \ref{difftheorem2} to establish our main result, Theorem \ref{mainresult}. Recall the definition of $F(\mathfrak{q})$ in \eqref{Fdef}. The set 
$$
\Lambda(q):=\left\{\left(\frac{\sigma_1(p)}{\sigma_2(q)},\frac{\sigma_2(p)}{\sigma_2(q)}\right) : p\in \mathcal{O}\right\}
$$
is a lattice in $\mathbb{R}^2$ which has a fundamental parallelogram whose area equals $1/\mathcal{N}(\mathfrak{q})$ and whose side lengths are $\asymp 1/\sqrt{\mathcal{N}(\mathfrak{q})}$, by the same arguments as at the end of  section \ref{Psiev}. If 
$$
\mathcal{N}(\mathfrak{q})\le N^{1+\varepsilon} \quad \mbox{ and } 
\quad \delta\le 
N^{-2\varepsilon},
$$
and $N$ is large enough, then taking into account the exponential decay of $\Omega_{\delta/\sqrt{N}}(y)$, we have
$$
F(\mathfrak{q})\ll 1
$$ 
if there exists $p\in \mathcal{O}$ such that
\begin{equation} \label{finalasy}
\left|x-\frac{\sigma_i(p)}{\sigma_i(q)}\right|\le \frac{\delta}{N^{1/2-\varepsilon}} \quad \mbox{ for } i=1,2,
\end{equation}
and $F(\mathfrak{q})$ is negligible otherwise. Using this together with $\eqref{Psisize}$ and the definition of $\tilde{\omega}(\mathfrak{q})$ in \eqref{2.1}, it follows that $\tilde{\mathcal{T}}(N)/\log N$ is dominated by the number of prime ideals $\mathfrak{p}\in \mathcal{I}\setminus 0$ satisfying 
$$
\mathcal{N}(\mathfrak{p})\le N^{1+\varepsilon}
$$
and \eqref{finalasy} for some $q$ generating $\mathfrak{p}$ and $p\in \mathcal{O}$, provided $\tilde{\mathcal{T}}(N)\gg 1$ and $\delta\le N^{-2\varepsilon}$ which latter we may assume without loss of generality. This together with Theorem \ref{difftheorem} and Corollary \ref{difftheorem2} implies the main result upon re-defining $\varepsilon$.

\section{Appendix - Proof of Harman's sieve for quadratic fields} 
Now we prove Theorem \ref{T4.1}, our weighted version of Harman's asymptotic sieve for ideals in the ring of integers of a quadratic field (not necessarily real quadratic and not necessarily of class number 1). We follow closely the proof of a weighted integer version of Harman's sieve for imaginary quadratic number fields in \cite[section 7]{BaiTech}.

The following lemma, known as "cosmetic surgery" will be used for the separation of variables. 
\begin{Lemma}  \label{L3.2}
For any two distinct real numbers $\rho,\gamma >0 $ and $T\geq1$ one has 
\begin{align*}
   \Bigg|1_{\gamma<\rho}-\frac{1}{\pi}\int_{-T}^{T}e^{i\gamma t} \frac{\sin(\rho t)}{t}\Bigg| \ll \frac{1}{T|\gamma-\rho|} ,
\end{align*}
where the implied constant is absolute.
\end{Lemma}
\begin{proof}
See, for instance, \cite[Lemma 2.2]{Harmanbook}. 
\end{proof}

Now we begin with the proof of Theorem \ref{T4.1}, where we use the notations introduced in subsection \ref{nota} and the following notations.

\begin{itemize}
\item  For a general condition $(C)$, we write
$$
1_{(C)} :=\begin{cases} 1 & \mbox{ if } (C) \mbox{ is satisfied,}\\ 0 & \mbox{ if } (C) \mbox{ is not satisfied.}
\end{cases}
$$
\item If $M$ is a set, we write
$$
1_{M}(x) := \begin{cases} 1 \mbox{ if } x\in M,\\ 0 \mbox{ otherwise.} \end{cases}
$$
\end{itemize}
First we define the M\"obius function $\mu$ for ideals. Assume that $\mathfrak{a}$ is an ideal in $\mathcal{O}$ with prime ideal factorization
$$
\mathfrak{a}=\prod_{j=1}^{k} \mathfrak{p}_j^{\alpha_j}.
$$ 
Then we set
\begin{align*}
 \mu(\mathfrak{a}):= 
       \begin{cases}
       (-1)^k & \mbox{ if } \alpha_j = 1 \mbox{ for } j=1,2,...,k, \\
       0 & \mbox{ otherwise.}
     \end{cases}
\end{align*}
Since
$$
\sum\limits_{\mathfrak{d}|\mathfrak{a}} \mu(\mathfrak{d})= \begin{cases} 1 & \mbox{ if } \mathfrak{a}=\mathcal{O}\\ 0 & \mbox{ otherwise,}
\end{cases}
$$
we have
\begin{align} \label{4.4}
\begin{split}
    S(w,z) =\sum_{\mathfrak{b}\in \mathcal{I}\setminus{0}} w(\mathfrak{b})\sum_{\substack{\mathfrak{d}|\mathcal{P}(z)\\ \mathfrak{d}|\mathfrak{b}}}\mu(\mathfrak{d})
    =& \sum_{\mathfrak{d}|\mathcal{P}(z)} \mu(\mathfrak{d})\sum_{\mathfrak{a}\in\mathcal{I}\setminus{0}}w(\mathfrak{ad}).
\end{split}
\end{align}
Let 
\begin{align} \label{4.5}
\Delta(\mathfrak{d})=\sum_{\mathfrak{a}\in \mathcal{I}\setminus{0}}(\omega(\mathfrak{ad})-\Tilde{\omega}(\mathfrak{ad})).
\end{align}
Applying \eqref{4.4} for $w=\omega$ and $w=\Tilde{\omega}$ yields
\begin{align} \label{4.6}
\begin{split}
S(\omega,z)-S(\Tilde{\omega},z)=\sum_{\mathfrak{d}|\mathcal{P}(z)}\mu(\mathfrak{d})\Delta(\mathfrak{d})=&
\bigg\{\sum_{\substack{\mathfrak{d}|\mathcal{P}(z)\\ \mathcal{N}(\mathfrak{d})<M}} + \sum_{\substack{\mathfrak{d}|\mathcal{P}(z)\\ \mathcal{N}(\mathfrak{d})\geq M}}\bigg\}\mu(\mathfrak{d})\Delta(\mathfrak{d})\\
=& S^{\sharp} + S^{\flat}, \mbox{ say}.
\end{split}
\end{align}
Using \eqref{4.2} with $a_{\mathfrak{d}}= \mu(\mathfrak{d})1_{\mathfrak{d}|\mathcal{P}(z)}$, we deduce that $|S^{\sharp}| \leq Y$. Therefore, to prove the theorem, it suffices to show that 
\begin{align} \label{4.7}
|S^{\flat}|\ll Y(\log(xX))^3. 
\end{align} \par
The next step is to arrange $S^{\flat}$ into subsums according to the "sizes" of the prime factors in $\mathfrak{d}$ (where $\mathfrak{d}$ is the summation variable in \eqref{4.6}).
To have some notion of size, fix some total order $\prec$ on $\mathbb{P}(z)$ such that if $\mathcal{N}(\mathfrak{p}_2)< \mathcal{N}(\mathfrak{p}_1)$, then $\mathfrak{p}_2 \prec \mathfrak{p}_1$ (many such orders exist, and all will do equally well). 
Moreover, for $\mathfrak{p}\in \mathbb{P}(z)$ let 
\begin{align*}
\Pi(\mathfrak{p})=\prod_{\mathfrak{q}\prec\mathfrak{p}} \mathfrak{q}.
\end{align*}
Now take $ g:\mathcal{I}\longrightarrow\mathbb{C}$ to be any  function. We may group the terms of the sum
\begin{align*}
S=\sum_{\mathfrak{d}|\mathcal{P}(z)}\mu(\mathfrak{d})g(\mathfrak{d})
\end{align*}
according to the largest factor $ \mathfrak{p}_1$  of $\mathfrak{d}$ with respect to $\prec$, getting the identity
\begin{align} \label{4.8}
S = g(\mathcal{O})-\sum_{\mathfrak{p}_1\in \mathbb{P}(z)}\sum_{\mathfrak{d}|\Pi(\mathfrak{p}_1)} \mu(\mathfrak{d})g(\mathfrak{p}_1\mathfrak{d}).
\end{align}
Similarly, for the part $\sum_{\mathfrak{d}|\Pi(\mathfrak{p}_1)} \mu(\mathfrak{d})g(\mathfrak{p}_1\mathfrak{d})$, we have 
\begin{align} \label{4.9}
    \sum_{\mathfrak{d}|\Pi(\mathfrak{p}_1)} \mu(\mathfrak{d})g(\mathfrak{p}_1\mathfrak{d})= g(\mathfrak{p}_1)-\sum_{\mathfrak{p}_2\prec\mathfrak{p}_1}\sum_{\mathfrak{d}|\Pi(\mathfrak{p}_1)} \mu(\mathfrak{d})g(\mathfrak{p}_1\mathfrak{p}_2\mathfrak{d}).
\end{align}
Minding the innermost sum on the right-hand side above, it is obvious that the above identity can be iterated if so desired. To describe for which sub-sums iteration is beneficial, we decompose $\mathbb{P}(z)$ into
\begin{align*}
    \mathbb{P}(z)=&\{\mathfrak{p}\in \mathbb{P}(z):\mathcal{N}(\mathfrak{p}_1)>x^\mu\}\mathbin{\dot{\cup}}\{\mathfrak{p}\in \mathbb{P}(z):\mathcal{N}(\mathfrak{p}_1)\leq x^\mu\}\\
    =&\mathcal{P}_1\mathbin{\dot{\cup}}\mathcal{Q}_1,\ \mbox{say},
\end{align*}
and inductively for $s=2,3,...$, we define
\begin{align*}
  \mathcal{Q'}_s =& \{(\mathfrak{p}_1,...,\mathfrak{p}_{s-1},\mathfrak{p}_s)\in{\mathbb{P}(z)}^s:\mathfrak{p}_s \prec \mathfrak{p}_{s-1},(\mathfrak{p}_1,...,\mathfrak{p}_{s-1})\in \mathcal{Q}_{s-1}\}\\
  =&\mathcal{P}_s\mathbin{\dot{\cup}}\mathcal{Q}_s,\ \mbox{say},
\end{align*}
where
\begin{align*}
  \mathcal{P}_s =& \{(\mathfrak{p}_1,...,\mathfrak{p}_{s-1},\mathfrak{p}_s)
  \in\mathcal{Q'}_s:\mathcal{N}(\mathfrak{p}_1\mathfrak{p}_2\cdots \mathfrak{p}_s)>x^\mu\}
\end{align*}
and 
\begin{align*}
  \mathcal{Q}_s =& \{(\mathfrak{p}_1,...,\mathfrak{p}_{s-1},\mathfrak{p}_s)\in\mathcal{Q'}_s:\mathcal{N}(\mathfrak{p}_1\mathfrak{p}_2\cdots \mathfrak{p}_s)\leq x^\mu\}.
\end{align*}

Assuming that $g$ vanishes on arguments $\mathfrak{a}$ with $\mathcal{N}(\mathfrak{a})\leq x^\mu$, and on applying \eqref{4.8} and \eqref{4.9}, we have 
\begin{align*}
\begin{split}
S= & -\bigg(\sum_{\mathfrak{p_1}\in\mathcal{P}_1}+\sum_{\mathfrak{p}_1\in\mathcal{Q}_1}\bigg)\sum_{\mathfrak{d}|\Pi(\mathfrak{p}_1)}\mu(\mathfrak{d})g(\mathfrak{p}_1\mathfrak{d})\\
  =& -\sum_{\mathfrak{p}_1\in\mathcal{P}_1}\sum_{\mathfrak{d}|\Pi(\mathfrak{p}_1)}\mu(\mathfrak{d})g(\mathfrak{p}_1\mathfrak{d})+\sum_{(\mathfrak{p}_1,\mathfrak{p}_2)\in \mathcal{P}_2}\sum_{\mathfrak{d}|\Pi(\mathfrak{p}_2)}\mu(\mathfrak{d})g(\mathfrak{p}_1\mathfrak{p}_2\mathfrak{d})
  + \sum_{(\mathfrak{p}_1,\mathfrak{p}_2)\in \mathcal{Q}_2}\sum_{\mathfrak{d}|\Pi(\mathfrak{p}_2)}\mu(\mathfrak{d})g(\mathfrak{p}_1\mathfrak{p}_2\mathfrak{d}).
\end{split}
\end{align*}
On iterating this process - always applying \eqref{4.9} to the $\mathcal{Q}$-part - it transpires that 
\begin{align*}
\begin{split}
   S= &\sum_{s\leq t}(-1)^s\sum_{(\mathfrak{p}_1,...,\mathfrak{p}_s)\in\mathcal{P}_s}\sum_{\mathfrak{d}|\Pi(\mathfrak{p_s})}\mu(\mathfrak{d})g(\mathfrak{p}_1\mathfrak{p}_2\cdots \mathfrak{p}_s\mathfrak{d})\\
   & +(-1)^t\sum_{(\mathfrak{p}_1,...,\mathfrak{p}_t)\in\mathcal{Q}_t}\sum_{\mathfrak{d}|\Pi(\mathfrak{p_t})}\mu(\mathfrak{d})g(\mathfrak{p}_1\mathfrak{p}_2\cdots \mathfrak{p}_t\mathfrak{d})
\end{split}
\end{align*}
for any $t\in \mathbb{N}$. Since the product of $t$ prime ideals has norm greater than or equal to $2^t$, we have
\begin{align*}
    \mathcal{Q}_t=\emptyset \mbox{ for } t>\frac{\mu}{\log 2}\log x.
\end{align*}
Hence,
\begin{align*}
    S=\sum_{s\leq t}(-1)^s \sum_{(\mathfrak{p}_1,...,\mathfrak{p}_s)\in\mathcal{P}_s}\sum_{\mathfrak{d}|\Pi(\mathfrak{p}_s)}\mu(\mathfrak{d})g(\mathfrak{p}_1\mathfrak{p}_2\cdots \mathfrak{p}_s\mathfrak{d})
\end{align*}
for 
\begin{align} \label{4.10}
    t:=\Bigl\lfloor{\frac{\log x}{\log 2}}\Bigr\rfloor+1 \ll \log x.   
\end{align} \par
We apply this to $S^{\flat}$ with $g(\mathfrak{a})=\Delta(\mathfrak{a})1_{\{{\mathcal{N}(\mathfrak{a})\geq M}\}}$. Note that since $M>x^\mu$, we have $g(\mathfrak{a})=0$ for all $\mathcal{N}(\mathfrak{a})\leq x^\mu$, as was assumed in the above arguments. Thus,
\begin{align} \label{4.11}
    S^{\flat}=\sum_{s\leq t}  (-1)^sS^{\flat}(s),
\end{align}
where 
\begin{align*}
     S^{\flat}(s)=\sum_{\substack{(\mathfrak{p}_1,...,\mathfrak{p}_s)\in\mathcal{P}_s \\ \mathfrak{a}=\mathfrak{p}_1\cdots \mathfrak{p}_s}}\sum_{\substack{\mathfrak{d}|\Pi(\mathfrak{p_s})\\ \mathcal{N}(\mathfrak{a}\mathfrak{p}_s)\geq M}} \mu(\mathfrak{d})\Delta(\mathfrak{ad}).
\end{align*}
Another application of \eqref{4.9} gives 
\begin{align} \label{4.12}
\begin{split}
 S^{\flat}(s)=&\sum_{\substack{(\mathfrak{p}_1,...,\mathfrak{p}_s)\in\mathcal{P}_s \\ \mathfrak{a}=\mathfrak{p}_1\cdots \mathfrak{p}_s\\ \mathcal{N}(\mathfrak{a})\geq M}} \Delta(\mathfrak{a})-\sum_{\substack{(\mathfrak{p}_1,...,\mathfrak{p}_s)\in\mathcal{P}_s \\ \mathfrak{a}=\mathfrak{p}_1\cdots \mathfrak{p}_s}} \sum_{\mathfrak{p}\prec \mathfrak{p}_s} \sum_{\substack{\mathfrak{d}|\Pi(\mathfrak{p)} \\ \mathcal{N}(\mathfrak{apd})\geq M}} \mu(\mathfrak{d})\Delta(\mathfrak{apd}) \\
 =& S^{\flat}_{1}(s)-S^{\flat}_{2}(s), \ \mbox{say}. 
\end{split}
\end{align}
Given $\mathfrak{a}=\mathfrak{p}_1\cdots \mathfrak{p}_{s-1}\mathfrak{p}_s$ with 
\begin{align*}
(\mathfrak{p}_1,...,\mathfrak{p}_{s-1},\mathfrak{p}_s)\in \mathcal{P}_s \quad \mbox{and} \quad (\mathfrak{p}_1,...,\mathfrak{p}_{s-1})\in \mathcal{Q}_{s-1},
\end{align*}
and noting that $\mathcal{N}(\mathfrak{p}_s)\leq \mathcal{N}(\mathfrak{p}_1)<z=x^ \kappa $, we have
\begin{align*}
  x^\mu < \mathcal{N}(\mathfrak{a})=\mathcal{N}(\mathfrak{p}_1\cdots \mathfrak{p}_{s-1})\mathcal{N}(\mathfrak{p}_s)<x^\mu \cdot x^\kappa.
\end{align*}
Using this, we find that $S^{\flat}_{1}(s)$ can be expressed as 
\begin{align*}
 \mathop{\sum\sum}\limits_{\mathfrak{a},\mathfrak{b} \in \mathcal{I}\setminus{0}} a_\mathfrak{a} (\omega(\mathfrak{ab})-\Tilde{\omega}(\mathfrak{ab})),
\end{align*}
where the coefficients
\begin{align*}
a_\mathfrak{a}=1_{\{\mathcal{N}(\mathfrak{a})\geq M\}} 1_{\{\mathfrak{p}_1\cdots \mathfrak{p}_s: (\mathfrak{p}_1,...,\mathfrak{p}_s)\in \mathcal{P}_s\}}(\mathfrak{a})
\end{align*}
are only supported on $\mathfrak{a}$ with $x^\mu<\mathcal{N}(\mathfrak{a})<x^{\mu+\kappa}$. Hence by \eqref{4.3}, 
\begin{align} \label{4.13}
|S^{\flat}_{1}(s)|\leq Y. 
\end{align}

Moving on to $S^{\flat}_{2}(s)$, we expand the definition \eqref{4.5} of $\Delta$, getting
\begin{align*}
    S^{\flat}_{2}(s)=S^{\flat}_{2}(s,\omega)-S^{\flat}_{2}(s,\Tilde{\omega}),
\end{align*}
where
\begin{align*}
 S^{\flat}_{2}(s,w):=&\sum_{\substack{(\mathfrak{p}_1,...,\mathfrak{p}_s)\in\mathcal{P}_s \\ \mathfrak{a}=\mathfrak{p}_1\cdots \mathfrak{p}_s}} \sum_{\mathfrak{p}\prec \mathfrak{p}_s} \sum_{\substack{\mathfrak{d}|\Pi(\mathfrak{p)} \\ \mathcal{N}(\mathfrak{apd})\geq M}} \mu(\mathfrak{d})\sum_{\mathfrak{b}\in\mathcal{I}\setminus{0}}w(\mathfrak{abpd})\\
 =&\sum_{\substack{(\mathfrak{p}_1,...,\mathfrak{p}_s)\in\mathcal{P}_s \\ \mathfrak{a}=\mathfrak{p}_1\cdots \mathfrak{p}_s}}\sum_{\mathfrak{n}\in\mathcal{I}\setminus{0}} \sum_{\mathfrak{p}\prec \mathfrak{p}_s}\mathop{\sum\sum}\limits_{\substack{\mathfrak{b,d}\\ \mathfrak{d}|\Pi(\mathfrak{p})\\ \mathfrak{bpd=n} \\ \mathcal{N}(\mathfrak{apd})\geq M}}\mu(\mathfrak{d})w(\mathfrak{an}).
 \end{align*}
 In order to apply \eqref{4.3}, we must disentangle the variables $\mathfrak{a}$ and $\mathfrak{n}$ in the above summation. To this end, we split
 \begin{align} \label{4.14}
   \sum_{\mathfrak{p}\prec \mathfrak{p}_s} = \sum_{\substack{\mathfrak{p}\prec \mathfrak{p}_s\\\mathcal{N}(\mathfrak{p})=\mathcal{N}(\mathfrak{p}_s)}}+\sum_{\substack{\mathfrak{p}\prec \mathfrak{p}_s\\\mathcal{N}(\mathfrak{p})<\mathcal{N}(\mathfrak{p}_s)}} 
 \end{align}
 to obtain a decomposition 
 \begin{align} \label{4.15}
 S^{\flat}_{2}(s,w)=S_{2}^{\flat,=}(s,w)+S_{2}^{\flat,<}(s,w), \mbox{ say.} 
 \end{align}
 For $S_{2}^{\flat,<}(s,w)$ we have 
 \begin{align*}
 S_{2}^{\flat,<}(s,w)=\sum_{\substack{(\mathfrak{p}_1,...,\mathfrak{p}_s)\in\mathcal{P}_s \\ \mathfrak{a}=\mathfrak{p}_1\cdots \mathfrak{p}_s}}\sum_{\mathfrak{n}\in\mathcal{I}\setminus{0}} \sum_{\mathfrak{p}\prec \mathfrak{p}_s}\mathop{\sum\sum}\limits_{\substack{\mathfrak{b,d}\\ \mathfrak{d}|\Pi(\mathfrak{p})\\ \mathfrak{bpd=n}}} \mu(\mathfrak{d})\chi(\mathfrak{a,d,p},\mathfrak{p}_s) w(\mathfrak{an}),
 \end{align*}
 where 
 \begin{align*}
    \chi(\mathfrak{a,d,p},\mathfrak{p}_s)=1_{\{\mathcal{N}(\mathfrak{apd})\ge M\}}1_{\{\mathcal{N}(\mathfrak{p})<\mathcal{N}(\mathfrak{p}_s) \}},
 \end{align*}
 and the sum $S_{2}^{\flat,=}(s,w)$ can be expressed similarly but needs a little more care.\par
The first summation on the right-hand side of \eqref{4.14} contains at most one term because $\mathbb{K}$ is a quadratic extension of $\mathbb{Q}$
and hence for each $l$ there are at most two prime ideals with norm $l$.
We will write $\mathcal{P}_s'$ for the set of $(\mathfrak{p}_1,...,\mathfrak{p}_s)\in \mathcal{P}_s$ for which there is such a term, that is, some $\mathfrak{p}\prec\mathfrak{p}_s$ with $\mathcal{N}(\mathfrak{p})=\mathcal{N}(\mathfrak{p} _s)$.
Furthermore, let
\begin{align*}
    \mathbb{P}'(z):=\{\mathfrak{p}\in\mathbb{P}(z) : \mbox{ there exists }\mathfrak{p}_s \mbox{ such that } \mathfrak{p}\prec\mathfrak{p}_s \mbox{ and }\mathcal{N}(\mathfrak{p})=\mathcal{N}(\mathfrak{p}_s) \}.
\end{align*}
Then
\begin{align*}
     S_{2}^{\flat,=}(s,w):=\sum_{\substack{(\mathfrak{p}_1,...,\mathfrak{p}_s)\in\mathcal{P}_s' \\ \mathfrak{a}=\mathfrak{p}_1\cdots\mathfrak{p}_s}}\sum_{\mathfrak{n}\in\mathcal{I}\setminus{0}} \sum_{\mathfrak{p}\in\mathbb{P}'(z)}\mathop{\sum\sum}\limits_{\substack{\mathfrak{b,d}\\\mathfrak{d}|\Pi(\mathfrak{p})\\ \mathfrak{bpd=n}}} \mu(\mathfrak{d})\Tilde{\chi}(\mathfrak{a,d,p},\mathfrak{p}_s) w(\mathfrak{an}),
\end{align*}
where
\begin{align} \label{4.16}
\begin{split}
    \Tilde{\chi}(\mathfrak{a,d,p},\mathfrak{p}_s)=&
1_{\{{\mathcal{N}(\mathfrak{apd})\geq M}\}}1_{\{\mathcal{N}(\mathfrak{p})=\mathcal{N}(\mathfrak{p}_s)\}}\\
=&1_{\{{\mathcal{N}(\mathfrak{apd})\geq M}\}}1_{\{\mathcal{N}(\mathfrak{p})\leq\mathcal{N}(\mathfrak{p}_s)\}}
-\chi(\mathfrak{a,d,p},\mathfrak{p}_s).
\end{split}
\end{align}
We choose some real number $\rho$ with $|\rho|\leq 1/2$ and $\{M+\rho\}=1/2$, where $\{.\}$ denotes the fractional part. Then the condition $\mathcal{N}(\mathfrak{apd})\geq M$ is equivalent to $\log \mathcal{N}(\mathfrak{apd}) \geq \log(M+\rho)$ and 
\begin{align*}
    |\log \mathcal{N}(\mathfrak{apd})- \log(M+\rho)|\geq \log\frac{x+1}{x+1/2}\geq \frac{1}{3x}.
\end{align*}
Therefore, Lemma \ref{L3.2} shows that 
\begin{align*}
    1_{\{\mathcal{N}(\mathfrak{apd})\geq M\}}=1-\frac{1}{\pi}\int_{-T}^{T}\mathcal{N}(\mathfrak{apd})^{it}\sin{(t\log(M+\rho))}\frac{dt}{t}+ O\left(\frac{x}{T}\right)
\end{align*}
for every $T\geq 1$. Similarly,
\begin{align*}
    1_{\{\mathcal{N}(\mathfrak{p})<\mathcal{N}(\mathfrak{p}_s)\}} =&\frac{1}{\pi}\int_{-T}^{T}e^{it/2}e^{it\mathcal{N}(\mathfrak{p})}\sin{(t\mathcal{N}(\mathfrak{p}_s))}\frac{dt}{t}+O\left(\frac{1}{T}\right)
\end{align*}
and 
\begin{align*}
    1_{\{\mathcal{N}(\mathfrak{p})\leq \mathcal{N}(\mathfrak{p}_s)\}} =&\frac{1}{\pi}\int_{-T}^{T}e^{-it/2}e^{it\mathcal{N}(\mathfrak{p})}\sin{(t\mathcal{N}(\mathfrak{p}_s))}\frac{dt}{t}+O\left(\frac{1}{T}\right).
\end{align*}
Thus,
\begin{align} \label{4.17}
\begin{split}
  S_{2}^{\flat,<}(s,\omega)=&\frac{1}{\pi}\int_{-T}^{T}\mathop{\sum\sum}\limits_{\substack{\mathfrak{a},\mathfrak{n}\in \mathcal{I}\setminus{0}}}a_\mathfrak{a}(t)b_\mathfrak{n}(t)\omega(\mathfrak{an})\frac{dt}{t}\\-& \frac{1}{\pi^2}\int_{-T}^{T}\int_{-T}^{T}\mathop{\sum\sum}\limits_{\substack{\mathfrak{a},\mathfrak{n}\in \mathcal{I}\setminus{0}}}a_\mathfrak{a}(t,\tau)b_\mathfrak{n}(t,\tau)\omega(\mathfrak{an})\frac{d\tau}{\tau}\frac{dt}{t}\\ &+O\Bigg(\Bigg(\frac{x}{T}+\frac{1}{T}\int_{-T}^{T}|\sin{(\tau\log(M+\rho))}|\frac{d\tau}{\tau}\Bigg)\times\\ &  \Bigg(\sum_{\substack{(\mathfrak{p}_1,...,\mathfrak{p}_s)\in\mathcal{P}_s \\ \mathfrak{a}=\mathfrak{p}_1\cdots \mathfrak{p}_s}}\sum_{\mathfrak{n}\in\mathcal{I}\setminus{0}} 
\sum_{\mathfrak{p}\prec \mathfrak{p}_s}\mathop{\sum\sum}\limits_{\substack{\mathfrak{b}, \mathfrak{d}\\ \mathfrak{d}|\Pi(\mathfrak{p})\\ \mathfrak{bpd=n}}}  w(\mathfrak{an}) \Bigg)\Bigg) 
\end{split}
\end{align}
with coefficients 
\begin{align} \label{4.18}
\begin{split}
    a_{\mathfrak{a}}(t):= &
       \begin{cases}
       \sin{(t\mathcal{N}(\mathfrak{p}_s))} & \mbox{ if there exists } (\mathfrak{p}_1,...,\mathfrak{p}_s)\in\mathcal{P}_s \mbox{ such that } \mathfrak{a}=\mathfrak{p}_1
\cdots \mathfrak{p}_s,   \\
       0 & \mbox{ otherwise,}
     \end{cases}\\
    b_\mathfrak{n}(t):=&\sum_{\mathfrak{p}\in\mathbb{P}(z)}\mathop{\sum\sum}\limits_{\substack{\mathfrak{b},\mathfrak{d}\\ \mathfrak{d}|\Pi(\mathfrak{p})\\ \mathfrak{bpd=n}}}e^{it/2}e^{it\mathcal{N}(\mathfrak{p})}\mu(\mathfrak{d}),\\
    a_{\mathfrak{a}}(t,\tau):=& a_{\mathfrak{a}}(t)\mathcal{N}(\mathfrak{a})^{i\tau}\sin{(\tau\log(M+\rho))},\\
    b_{\mathfrak{n}}(t,\tau):=&\sum_{\mathfrak{p}\in\mathbb{P}(z)}\mathop{\sum\sum}\limits_{\substack{\mathfrak{b},\mathfrak{d}\\ 
\mathfrak{d}|\Pi(\mathfrak{p})\\ \mathfrak{bpd=n}}}e^{\frac{it}{2}}e^{it\mathcal{N}(\mathfrak{p})}\mu(\mathfrak{d})(\mathcal{N}\mathfrak{pd})^{i\tau}.
\end{split}
\end{align}\par
We proceed by gathering some intermediate information before applying \eqref{4.3}. Clearly, 
\begin{align*}
     |b_\mathfrak{n}(t)|,|b_{\mathfrak{n}}(t,\tau)|\leq d(\mathfrak{n}).
\end{align*}
For the other coefficients we always have
\begin{align*}
     |a_\mathfrak{a}(t)|,|a_{\mathfrak{a}}(t,\tau)|\leq 1,
\end{align*}
yet if $t$ and $\tau$ are small, we can do better: if $|t|\leq x^{-1/2}$ and $|\tau|\leq (\log(x+1/2))^{-1}$, then
\begin{align} \label{4.19}
    |a_\mathfrak{a}(t)|\leq\sqrt{x}|t|, \quad |a_{\mathfrak{a}}(t,\tau)|\leq\sqrt{x}|t\tau|\log\left(x+\frac{1}{2}\right). 
\end{align}
In view of this, we must deal with functions $f:\mathbb{R} \times (1,\infty)\longrightarrow \mathbb{R}$ of the shape
\begin{align*}
    f(t,\eta)=
       \begin{cases}
       \eta |t| \quad if\quad |t|\leq \eta^{-1},\\ 
       1 \quad \text{otherwise}\\
     \end{cases}
\end{align*}
and their integrals
\begin{align} \label{4.20}
    \int_{-T}^{T} f(t,\eta)\frac{dt}{|t|} \ll \eta \int_{0}^{\eta^{-1}} dt + \Bigg|\int_{\eta^{-1}}^T \frac{dt}{t}\Bigg|\ll 1+|\log{(T\eta)}|. 
\end{align}
Lastly, we note that by \eqref{4.1},
\begin{align} \label{4.21}
   \sum_{\substack{(\mathfrak{p}_1,...,\mathfrak{p}_s)\in\mathcal{P}_s \\ \mathfrak{a}=\mathfrak{p}_1\cdots \mathfrak{p}_s}}\sum_{\mathfrak{n}\in\mathcal{I}\setminus{0}} \sum_{\mathfrak{p}\prec \mathfrak{p}_s}\mathop{\sum\sum}\limits_{\substack{\mathfrak{b},\mathfrak{d}\\ \mathfrak{d}|\Pi(\mathfrak{p})\\ \mathfrak{bpd=n}}}  w(\mathfrak{an}) \ll
   \sum_{\mathfrak{a}\in \mathcal{I}\setminus{0}} d_4(\mathfrak{a})w(\mathfrak{a})\ll X. 
\end{align}\par
Gathering all information we got so far, we may derive a bound for 
\begin{align*}
    \mathcal{E}^{<} = \left|S_{2}^{\flat,<}(s,\omega)-S_{2}^{\flat,<}(s,\Tilde{\omega})\right|
\end{align*}
as follows: after applying \eqref{4.17} with $w=\omega$ and $w=\Tilde{\omega}$, the $O$-terms are treated directly with \eqref{4.20} and \eqref{4.21},
whereas for the rest one may apply \eqref{4.3}. Here it is important to use \eqref{4.19} for small $|t|$ and $|\tau|$ first - prior to applying \eqref{4.3} - and
\eqref{4.20} then bounds the integrals. Therefore, after some computations, we infer 
\begin{align}\label{4.22}
\begin{split}
     \mathcal{E}^{<} \ll  Y\log(Tx)\left(1+\log\left(T\log\left(x+1/2\right)\right)\right)
     + XT^{-1}\left(x+\log\left(T\log\left(x+1/2\right)\right)\right). 
\end{split}
\end{align}
Of course, the same arguments also apply to  
$$
\mathcal{E}^{=} = \left|S_{2}^{\flat,=}(s,\omega)-S_{2}^{\flat,=}(s,\Tilde{\omega})\right|:
$$
In view of \eqref{4.16}, we have to apply them twice, but in both cases the coefficients corresponding to \eqref{4.18} obey the same bounds we used to derive \eqref{4.22}. Consequently, \eqref{4.22} also holds with $S_{2}^{\flat,=}$ in place of $S_{2}^{\flat,<}$. In total, recalling \eqref{4.12} , \eqref{4.13} and \eqref{4.15}, we have 
\begin{align*}
  |S^{\flat}(s)|\ll Y + \mbox{the bound from \eqref{4.22}}
\end{align*}
and it transpires that choosing $T=xX$ suffices to yield a bound of $\ll Y(\log(xX))^2$. On plugging this into \eqref{4.11} and recalling \eqref{4.10}, we infer \eqref{4.7}. Hence, the theorem is proved. $\Box$

\end{document}